\documentclass[12pt, a4paper, twoside,leqno]{amsart}
\usepackage[centering, totalwidth = 405pt, totalheight = 600pt]{geometry}
\usepackage{amssymb, amsmath, amsthm, enumerate, microtype, stmaryrd, url}
\usepackage[latin1]{inputenc}
\usepackage[arrow, matrix, tips, curve]{xy}
\usepackage{color}
\definecolor{darkgreen}{rgb}{0,0.45,0}
\usepackage[pagebackref,colorlinks,citecolor=darkgreen,linkcolor=darkgreen]{hyperref}
\SelectTips{cm}{10}

 \DeclareMathOperator{\ob}{ob}

\newcommand{\cat}[1]{\mathbf{#1}}
\newcommand{\op}{\mathrm{op}}
\newcommand{\id}{\mathrm{id}}
\newcommand{\thg}{{\mathord{\text{--}}}}

\newcommand{\abs}[1]{{\left|{#1}\right|}}

\newcommand{\defeq}{\mathrel{\mathop:}=}

\newcommand{\cd}[2][]{\vcenter{\hbox{\xymatrix#1{#2}}}}

\renewcommand{\phi}{\varphi}

\newcommand{\B}{{\mathcal B}}
\newcommand{\C}{{\mathcal C}}
\newcommand{\D}{{\mathcal D}}
\newcommand{\E}{{\mathcal E}}
\newcommand{\F}{{\mathcal F}}

\newcommand{\K}{{\mathcal K}}
\renewcommand{\L}{{\mathcal L}}
\newcommand{\M}{{\mathcal M}}

\renewcommand{\O}{{\mathcal O}}
\renewcommand{\P}{{\mathcal P}}

\newcommand{\V}{{\mathcal V}}
\newcommand{\W}{{\mathcal W}}

\newcommand{\Y}{{\mathcal Y}}

\renewcommand{\o}{\O}
\newcommand{\om}[1]{{#1^\ast}}
\newcommand{\xtor}[1]{\cd[@1]{{} \ar[r]|-{\object@{|}}^{#1} & {}}}
\newcommand{\tor}{\ensuremath{\relbar\joinrel\mapstochar\joinrel\rightarrow}}

\makeatletter

\def\hookleftarrowfill@{\arrowfill@\leftarrow\relbar{\relbar\joinrel\rhook}}
\def\twoheadleftarrowfill@{\arrowfill@\twoheadleftarrow\relbar\relbar}
\def\leftbararrowfill@{\arrowdoublefill@{\leftarrow\mkern-5mu}\relbar\mapstochar\relbar\relbar}
\def\Leftbararrowfill@{\arrowdoublefill@{\Leftarrow\mkern-2mu}\Relbar\Mapstochar\Relbar\Relbar}
\def\leftringarrowfill@{\arrowdoublefill@{\leftarrow\mkern-3mu}\relbar{\mkern-3mu\circ\mkern-2mu}\relbar\relbar}
\def\lefttriarrowfill@{\arrowfill@{\mathrel\triangleleft\mkern0.5mu\joinrel\relbar}\relbar\relbar}
\def\Lefttriarrowfill@{\arrowfill@{\mathrel\triangleleft\mkern1mu\joinrel\Relbar}\Relbar\Relbar}

\def\hookrightarrowfill@{\arrowfill@{\lhook\joinrel\relbar}\relbar\rightarrow}
\def\twoheadrightarrowfill@{\arrowfill@\relbar\relbar\twoheadrightarrow}
\def\rightbararrowfill@{\arrowdoublefill@{\relbar\mkern-0.5mu}\relbar\mapstochar\relbar\rightarrow}
\def\Rightbararrowfill@{\arrowdoublefill@{\Relbar\mkern-2mu}\Relbar\Mapstochar\Relbar\Rightarrow}
\def\rightringarrowfill@{\arrowdoublefill@\relbar\relbar{\mkern-2mu\circ\mkern-3mu}\relbar{\mkern-3mu\rightarrow}}
\def\righttriarrowfill@{\arrowfill@\relbar\relbar{\relbar\joinrel\mkern0.5mu\mathrel\triangleright}}
\def\Righttriarrowfill@{\arrowfill@\Relbar\Relbar{\Relbar\joinrel\mkern1mu\mathrel\triangleright}}

\def\leftrightarrowfill@{\arrowfill@\leftarrow\relbar\rightarrow}
\def\mapstofill@{\arrowfill@{\mapstochar\relbar}\relbar\rightarrow}

\renewcommand*\xleftarrow[2][]{\ext@arrow 20{20}0\leftarrowfill@{#1}{#2}}
\providecommand*\xLeftarrow[2][]{\ext@arrow 60{22}0{\Leftarrowfill@}{#1}{#2}}
\providecommand*\xhookleftarrow[2][]{\ext@arrow 10{20}0\hookleftarrowfill@{#1}{#2}}
\providecommand*\xtwoheadleftarrow[2][]{\ext@arrow 60{20}0\twoheadleftarrowfill@{#1}{#2}}
\providecommand*\xleftbararrow[2][]{\ext@arrow 10{22}0\leftbararrowfill@{#1}{#2}}
\providecommand*\xLeftbararrow[2][]{\ext@arrow 50{24}0\Leftbararrowfill@{#1}{#2}}
\providecommand*\xleftringarrow[2][]{\ext@arrow 10{26}0\leftringarrowfill@{#1}{#2}}
\providecommand*\xlefttriarrow[2][]{\ext@arrow 80{24}0\lefttriarrowfill@{#1}{#2}}
\providecommand*\xLefttriarrow[2][]{\ext@arrow 80{24}0\Lefttriarrowfill@{#1}{#2}}

\renewcommand*\xrightarrow[2][]{\ext@arrow 01{20}0\rightarrowfill@{#1}{#2}}
\providecommand*\xRightarrow[2][]{\ext@arrow 04{22}0{\Rightarrowfill@}{#1}{#2}}
\providecommand*\xhookrightarrow[2][]{\ext@arrow 00{20}0\hookrightarrowfill@{#1}{#2}}
\providecommand*\xtwoheadrightarrow[2][]{\ext@arrow 03{20}0\twoheadrightarrowfill@{#1}{#2}}
\providecommand*\xrightbararrow[2][]{\ext@arrow 01{22}0\rightbararrowfill@{#1}{#2}}
\providecommand*\xRightbararrow[2][]{\ext@arrow 04{24}0\Rightbararrowfill@{#1}{#2}}
\providecommand*\xrightringarrow[2][]{\ext@arrow 01{26}0\rightringarrowfill@{#1}{#2}}
\providecommand*\xrighttriarrow[2][]{\ext@arrow 07{24}0\righttriarrowfill@{#1}{#2}}
\providecommand*\xRighttriarrow[2][]{\ext@arrow 07{24}0\Righttriarrowfill@{#1}{#2}}

\providecommand*\xmapsto[2][]{\ext@arrow 01{20}0\mapstofill@{#1}{#2}}
\providecommand*\xleftrightarrow[2][]{\ext@arrow 10{22}0\leftrightarrowfill@{#1}{#2}}
\providecommand*\xLeftrightarrow[2][]{\ext@arrow 10{27}0{\Leftrightarrowfill@}{#1}{#2}}

\makeatother

\newcommand{\twocong}[2][0.5]{\ar@{}[#2] \save ?(#1)*{\cong}\restore}
\newcommand{\twoeq}[2][0.5]{\ar@{}[#2] \save ?(#1)*{=}\restore}
\newcommand{\rtwocell}[3][0.5]{\ar@{}[#2] \ar@{=>}?(#1)+/l 0.2cm/;?(#1)+/r 0.2cm/^{#3}}
\newcommand{\ltwocell}[3][0.5]{\ar@{}[#2] \ar@{=>}?(#1)+/r 0.2cm/;?(#1)+/l 0.2cm/^{#3}}
\newcommand{\ltwocello}[3][0.5]{\ar@{}[#2] \ar@{=>}?(#1)+/r 0.2cm/;?(#1)+/l 0.2cm/_{#3}}
\newcommand{\dtwocell}[3][0.5]{\ar@{}[#2] \ar@{=>}?(#1)+/u  0.2cm/;?(#1)+/d 0.2cm/^{#3}}
\newcommand{\dltwocell}[3][0.5]{\ar@{}[#2] \ar@{=>}?(#1)+/ur  0.2cm/;?(#1)+/dl 0.2cm/^{#3}}
\newcommand{\drtwocell}[3][0.5]{\ar@{}[#2] \ar@{=>}?(#1)+/ul  0.2cm/;?(#1)+/dr 0.2cm/^{#3}}
\newcommand{\dthreecell}[3][0.5]{\ar@{}[#2] \ar@3{->}?(#1)+/u  0.2cm/;?(#1)+/d 0.2cm/^{#3}}
\newcommand{\utwocell}[3][0.5]{\ar@{}[#2] \ar@{=>}?(#1)+/d 0.2cm/;?(#1)+/u 0.2cm/_{#3}}
\newcommand{\dtwocelltarg}[3][0.5]{\ar@{}#2 \ar@{=>}?(#1)+/u  0.2cm/;?(#1)+/d 0.2cm/^{#3}}
\newcommand{\utwocelltarg}[3][0.5]{\ar@{}#2 \ar@{=>}?(#1)+/d  0.2cm/;?(#1)+/u 0.2cm/_{#3}}

\newdir{(}{{}*!<0em,-.14em>-\cir<.14em>{l^r}}
\newdir{ (}{{}*!/-5pt/\dir{(}}
\newdir{ >}{{}*!/-5pt/\dir{>}}

\theoremstyle{definition}

\swapnumbers
\theoremstyle{plain}
\newtheorem{Thm}[subsection]{Theorem}
\newtheorem{Prop}[subsection]{Proposition}
\newtheorem{Cor}[subsection]{Corollary}
\newtheorem{Lemma}[subsection]{Lemma}

\numberwithin{equation}{section}

\theoremstyle{definition}

\newtheorem{Exs}[subsection]{Examples}
\newtheorem{Rk}[subsection]{Remark}

\newcommand{\Lan}{\mathrm{Lan}}

\begin{document}
 \leftmargini=2em
\title[Restriction categories as enriched categories]{Restriction categories\\as enriched categories}
\author{Robin Cockett}
\address{Department of Computer Science, University of Calgary, 2500 University Dr. NW, Calgary, Alberta, Canada T2N 1N4}
\email{robin@ucalgary.ca}
\author{Richard Garner}
\address{Department of Computing, Macquarie University, North Ryde, NSW 2109, Australia}
\email{richard.garner@mq.edu.au}
% \subjclass[2000]{Primary: ; Secondary: }
\date{\today}
\thanks{Partially supported by NSERC Canada and by the Australian Research Council (Discovery Projects Scheme, grant number DP110102360).}
\begin{abstract}
Restriction categories were introduced to provide an axiomatic setting for the study of partially defined mappings; 
they are categories equipped with an operation called \emph{restriction} which assigns to every morphism an endomorphism of its domain,
to be thought of as the partial identity that is defined to just the same degree as the original map. In this paper, we show that
restriction categories can be identified with \emph{enriched categories} in the sense of Kelly for a suitable enrichment base. By
varying that base appropriately, we are also able to capture the notions of \emph{join} and \emph{range} restriction category in terms of enriched
category theory.
\end{abstract}

\maketitle

\section{Introduction}
The notion of \emph{restriction category} was introduced in~\cite{Cockett2002Restriction} and provides an axiomatic setting for the study of notions of partiality. A restriction category is a category equipped with an operation which, to every map $f \colon A \to B$ assigns a map $\bar f \colon A \to A$, subject to four axioms designed to capture the following intuition: that the morphisms of a restriction category represent partially defined maps, with the map $\bar f$ being the partial identity map of $A$ which is defined to just the same degree as $f$ is. This structure, whilst very simple, supports a theory of partiality rich enough to capture results from computability theory, algebraic geometry, differential geometry, and the theory of inverse semigroups; see~\cite{Cockett2011Differential,Cockett2008Introduction,Cockett2002Restriction}.

A restriction category is a particular kind of structured category, and so many aspects of ordinary category theory have ``restriction'' correlates: thus there are notions of restriction functors and natural transformation, of restriction limits and colimits, and so on. The generalisation from ordinary categories to restriction categories requires some thought; so, for example, whilst limits and colimits in ordinary categories behave in a completely dual manner, the same is not true of restriction limits and restriction colimits---essentially because the notion of restriction category is not self-dual. It is therefore reasonable to ask how one may justify the validity of these generalisations.

In this article, we answer this question by exhibiting restriction categories as a particular kind of \emph{enriched category} in the sense of~\cite{Kelly1982Basic,Walters1981Sheaves}. This means that aspects of the theory of restriction categories can be read off from the corresponding aspects of enriched category theory. Whilst we shall not do this here, we will in subsequent work exploit these observations to define \emph{weighted} restriction limits and colimits, using them to exhibit categories of sheaves, of schemes, of manifolds, and so on, as ``free cocompletions in the restriction world''. 

We now give a brief overview of the contents of this paper. In Section~\ref{sec:2} we recall the basic restriction notions, 
and introduce the (localic, hyperconnected) factorisation system\footnote{While this factorization has clear categorical precedents, we were not able to find reference to it in the inverse or restriction semigroup literature.} on restriction functors. The hyperconnected restriction functors are, intuitively, those which reflect as well as preserve the restriction structure; the localic morphisms are abstractly characterised as those orthogonal on the left to the hyperconnected ones. However, we are able to provide an explicit description of the localic morphisms, and also of the (localic, hyperconnected) factorisation of a restriction functor.  

In Section~\ref{sec:fund}, we recall the \emph{fundamental functor} associated to every restriction category $\C$; this is a canonical restriction functor from $\C$ into a particular restriction category $\cat{Stab}^\op$. We give a universal characterisation of the fundamental functor by showing that it is, to within isomorphism, the unique hyperconnected functor $\C \to \cat{Stab}^\op$.

In Section~\ref{sec:3}, we recall the construction which assigns a $2$-category $\Gamma \C$ to any restriction category $\C$; we see that the hyperconnected restriction functors $\C \to \D$ correspond to the $2$-functors $\Gamma \C \to \Gamma \D$ which are local discrete fibrations (discrete fibrations on each hom-category). Combining this with the results of Section~\ref{sec:2}, we show that the definition of restriction category can be recast in purely $2$-categorical terms: a restriction category corresponds to a local discrete fibration of $2$-categories whose codomain is $\Gamma(\cat{Stab}^\op)$.

In Section~\ref{sec:enrichment}, we break off briefly to describe the appropriate notions of enrichment required for our main result; and then in Section~\ref{sec:4}, we give that result, exhibiting restriction categories, functors and natural transformations as enriched categories, functors and natural transformations over a suitable enrichment base. The key result we require is one due to Richard Wood in collaboration with the first author, characterising local discrete fibrations over a given $2$-category as categories enriched in an associated bicategory.

Finally, in Sections~\ref{sec:join} and~\ref{sec:range}, we give two variations on our main result, by considering \emph{join} restriction categories---ones in which families of compatible partial maps admits patchings---and \emph{range} restriction categories---ones in which every map has a ``codomain'' as well as a ``domain'' of definition---and exhibiting both kinds of structure as enriched categories over some suitably modified base.

\section{Restriction categories and the (localic, hyperconnected) factorisation}\label{sec:2}

We begin by recalling from~\cite{Cockett2002Restriction} some basic facts and results. A \emph{restriction category} is a category $\C$ equipped with an operation assigning to each map $f \colon A \to B$ in $\C$ a map $\bar f \colon A \to A$, subject to the four axioms:
\begin{enumerate}[(R1)]
\item $f \bar f = f$ for all $f \colon A \to B$;
\item $\bar f \bar g = \bar g \bar f$ for all $f \colon A \to B$ and $g \colon A \to C$;
\item $\overline{g \bar f} = \bar g \bar f$ for all $f \colon A \to B$ and $g \colon A \to C$;
\item $\bar g f = f \overline{gf}$ for all $f \colon A \to B$ and $g \colon B \to C$.
\end{enumerate}

A functor $F \colon \C \to \D$ between restriction categories is a \emph{restriction functor} if $F\bar f = \overline{Ff}$ for all $f \colon A \to B$ in $\C$. The restriction categories and restriction functors are the objects and $1$-cells of a $2$-category $\cat{rCat}$, whose $2$-cells are \emph{total} natural transformations; a map $f \colon A \to B$ in a restriction category is called \emph{total} if $\bar f = 1_A$, and a natural transformation is called total if all its components are. 

If $\C$ is a restriction category, we may partially order each of its homsets by taking $f \leqslant g$ just when $g \bar f = f$; the informal meaning being that $f$ is obtained from $g$ by restricting  it to some smaller domain of definition. This partial ordering is preserved by composition, and also by the action on homs of any restriction functor.

The examples that follow are drawn from~\cite[\S 2.1.3]{Cockett2002Restriction}.
\begin{Exs}\label{exs:restriction}\hfill
\begin{enumerate}[(i)]
\item The category $\cat{Set}_p$ of sets and partial functions is a restriction category, where to each partial function $f \colon A \rightharpoonup B$ we assign the partial function $\bar f \colon A \rightharpoonup A$ defined by taking $\bar f(a)$ to be $a$ if $f(a)$ is defined, and to be undefined otherwise.
\vskip0.5\baselineskip
\item The category $\cat{Top}_p$, of topological spaces and continuous functions defined on some open subset of their domain, is a restriction category. The restriction structure is given as in (i).
\vskip0.5\baselineskip
\item Generalising (i) and (ii), let $\D$ be a category equipped with a class $\M$ of monics which is closed under composition, stable under pullback, and contains the identities. There is a restriction category $\cat{Par}(\D, \M)$
with objects those of $\D$, and as morphisms $X \tor Y$, isomorphism-classes of spans $f \colon X \leftarrow Z \rightarrow Y \colon g$ with left leg in $\M$. The restriction of such a span $(f,g)$ is $(f,f) \colon X \tor X$.
\vskip0.5\baselineskip
\item The category $\cat{Rec}$ with objects, the natural numbers, and morphisms $n \to m$, partial recursive functions $\mathbb N^n \to \mathbb N^m$, is a restriction category; the restriction structure is as in (i), bearing in mind that $\bar f$ will be partial recursive whenever $f$ is.  $\cat{Rec}$ is the canonical example of a {\em Turing category\/}: for more details, see \cite{Cockett2008Introduction}.
\vskip0.5\baselineskip
\item An \emph{inverse monoid} is a monoid $M$ such that, for each $x \in M$, there exists a unique element $x^*$, the \emph{partial inverse} of $x$, satisfying $xx^*x = x$ and $x^* x x^* = x^*$; the basic example is the set of injective partial endofunctions of some set $A$. An inverse monoid can be seen as a one object restriction category where the restriction structure is given by $\bar x = x x^*$.  More generally, an \emph{inverse category} is a restriction category in which every map $x$ is a \emph{partial isomorphism}, meaning that there is a map $x^*$ such that $\bar x = x^* x$ and $\overline{x^*} = x x^*$; such partial inverses can be shown to be unique if they exist. Inverse categories stand in the same relationship to restriction categories as groupoids do to ordinary categories; a one object inverse category is an inverse monoid.

\vskip0.5\baselineskip
\item Consider the category $\cat{Stab}$ whose objects are meet-semilattices, and whose morphisms $A \to B$ are \emph{stable maps}---monotone maps which preserve binary meets, but not necessarily the top element. The category $\cat{Stab}^\op$ is a restriction category under the structure which to a stable map $f \colon A \leftarrow B$ assigns the morphism $\bar f \colon A \leftarrow A$ given by
$\bar f(a) = a \wedge f(\top)$.
\end{enumerate}
\end{Exs}

In a restriction category $\C$, each map of the form $\bar f \colon A \to A$ is idempotent, and we call them \emph{restriction idempotents}; they are equally well the endomorphisms $e$ satisfying $\bar e = e$. We write $\O(A)$ for the set of restriction idempotents on $A$. When $\C = \cat{Set}_p$,  $\O(A)$ is isomorphic to the power-set of $A$; when $\C = \cat{Top}_p$,  $\O(A)$ is isomorphic to the open-set lattice of $A$; when $\C = \cat{Par}(\D, \M)$, $\O(A)$ is the set of $\M$-subobjects of $A$; when $\C = \cat{Rec}$, $\O(n)$ is isomorphic to the set of recursively enumerable subsets of $\mathbb N^n$; and when $\C = \cat{Stab}^\op$, $\O(A)$ is isomorphic to $A$ itself.

We now introduce a class of restriction functors which will play an important role in what follows. Any restriction functor $F \colon \C \to \D$ sends restriction idempotents to restriction idempotents, and so induces, for each $A \in \C$ a mapping $\O(A) \to \O(FA)$. We define $F$ to be \emph{hyperconnected} if each such mapping is an isomorphism.  Our terminology is drawn from topos theory, where a geometric morphism $f \colon \E \to \F$  is called hyperconnected if its inverse image functor $f^\ast \colon \F \to \E$ induces isomorphisms of subobject lattices $\cat{Sub}_\F(A) \to \cat{Sub}_\E(f^\ast A)$ for every $A \in \E$. We can make the analogy precise: to each a topos $\E$ we can associate the restriction category $\cat{Par}(\E, \M)$ as in Examples~\ref{exs:restriction}(iii), with $\M$ the class of all monomorphisms in $\E$; 
and to each geometric morphism $f \colon \E \to \F$, we can, via its inverse image part, associate a restriction functor $\cat{Par}(\F, \M) \to \cat{Par}(\E, \M)$. This restriction functor will be hyperconnected just when the original $f$ is a hyperconnected geometric morphism.

In the topos-theoretic context, we have a factorisation system (hyperconnected, localic); in the restriction setting, we have a corresponding factorisation (localic, hyperconnected). Note the reversal of the two classes: this is because restriction functors point in the ``algebraic'' direction whereas topos morphisms point in the opposite, ``geometric'' direction. In the restriction setting, we define a restriction functor $F \colon \C \to \D$ to be \emph{localic} just when it is bijective on objects, and for every map $g \colon FA \to FB$ in $\D$, the poset of maps $f \colon A \to B$ in $\C$ with $g \leqslant Ff$ is downwards-directed (in particular, nonempty). 

\begin{Prop}\label{prop:localic-hyper-orthogonal}
Localic and hyperconnected restriction functors are orthogonal.
\end{Prop}
\begin{proof}
Given a commutative square in $\cat{rCat}$ as in
\begin{equation*}
\cd{
\C \ar[r]^H \ar[d]_F & \E \ar[d]^G \\
\D \ar[r]_K \ar@{.>}[ur]_J & \F 
}
\end{equation*}
with $F$ localic and $G$ hyperconnected, we must show that there is a unique filler $J$ as indicated making both triangles commute.
We do so at the level of objects by taking $JX = H\tilde X$, where $\tilde X$ is the unique object of $\C$ with $F\tilde X = X$.
On morphisms, given $f \colon X \to Y$ in $\D$, we note first that  $KX = KF\tilde X = GH\tilde X$, so that $\overline{Kf} \in \O(GH\tilde X)$; now since $G$ is hyperconnected, there is a unique $e \in \O(H \tilde X)$ with $Ge = \overline{Kf}$. Furthermore, since $F$ is localic, there exists a morphism $h \colon \tilde X \to \tilde Y$ in $\C$ with $f \leqslant Fh$, and we now define $Jf = Hh . e$. 

Note that $f \leqslant Fh$ implies $Kf \leqslant KFh$, whence $GJf = GHh.Ge =  KFh.\overline{Kf} = Kf$, showing that the lower-right triangle commutes.
Now $G(\overline{Jf}) = \overline{GJf} = \overline{Kf} = Ge$, whence by hyperconnectedness, $\overline{Jf} = e$.
It follows that the definition of $Jf$ is independent of the choice of $h$; for if $h'$ is another map with
$f \leqslant Fh'$, then by directedness, there exists $h'' \leqslant h, h'$ with $f \leqslant Fh''$, and by symmetry, it now suffices to show that $Hh.e = Hh''.e$. But both are $\leqslant Hh$, and both have the same restriction $e$, and so must coincide. 
It also follows that the upper triangle commutes; for when $f = Fg$ above, we have $e = \overline{Hg}$, and may take $h = g$, whence $JFg = Hg.\overline{Hg} = Hg$, as required.

We now show that $J$ is functorial. When $f = 1_X$ above, we have $e = 1_{JX}$, and may take $h = 1_{\tilde X}$; whence $J(1_X) = 1_{JX}$ as required. When $f = f_1. f_2$, with $Jf_1 = Hh_1.e_1$ and $Jf_2 = Hh_2.e_2$, say, then we have $J(f_1.f_2) = Hh_1.Hh_2.e$, where $Ge = \overline{Kf_1.Kf_2}$. But by~\cite[Lemma 2.1(iii)]{Cockett2002Restriction},  
$\overline{Kf_1.Kf_2} = \overline{\overline{Kf_1}.Kf_2} = \overline{Ge_1.Kf_2} = \overline{Ge_1.GJf_2} = G(\overline{e_1.Jf_2})$, whence $e = \overline{e_1.Jf_2}$ by hyperconnectedness. So now
$J(f_1.f_2) = Hh_1.Hh_2.e = Hh_1.Hh_2.\overline{e_1.Jf_2} = Hh_1.Hh_2.\overline{e_1.Hh_2.e_2} = Hh_1.Hh_2.\overline{e_1.Hh_2}.e_2 = Hh_1.e_1.Hh_2.e_2 = Jf_1.Jf_2$ as required.

Finally, we must verify that $J$ is the \emph{unique} diagonal filler for this square. Suppose that $J'$ were another such. Clearly $JX = J'X$ on objects; on morphisms, given $f \colon X \to Y$ in $\D$ as above, we choose $h$ with $f \leqslant Fh$, and now $Jf \leqslant JFh = Hh$ and $J'f \leqslant J'Fh = Hh$. But
$G(\overline{J'f}) = \overline{Kf} = Ge$ implies $\overline{J'f} = e = \overline{Jf}$ and so $Jf = Hh.\overline{Jf} = Hh.\overline{J'f} = J'f$, as required.
\end{proof}

\begin{Prop}\label{prop:loc-hyp-fact}
$\cat{rCat}$ admits (localic, hyperconnected) factorisations.
\end{Prop}
\begin{proof}
We need to construct a factorisation 
\begin{equation*}
\C \xrightarrow H \E \xrightarrow K \D
\end{equation*}
for any restriction functor $F \colon \C \to \D$. We take the category $\E$ to have the same objects as $\C$, and morphisms $x \to y$ being equivalence classes of pairs $(f, g)$ where $f \colon x \to y$ in $\C$ and $g \colon Fx \to Fy$ in $\D$ with $g \leqslant Ff$; the equivalence relation relates $(f,g)$ and $(f',g)$ just when there exists a pair $(f'', g)$ with $f'' \leqslant f, f'$. The restriction of the equivalence class $[f, g]$ is $[\bar f, \bar g] = [1_x, \bar g]$; the functor $H$ is the identity on objects and on morphisms sends $f$ to $[f, Ff]$; whilst $K$ acts as $F$ does on objects, and on morphisms sends $[f,g]$ to $g$. The remaining details are straightforward.
\end{proof}

In fact, the localic and hyperconnected restriction functors enjoy a stronger orthogonality property than that described above, by virtue of the following result.
\begin{Prop}\label{prop:codisccofib}
Each localic morphism is a codiscrete cofibration in $\cat{rCat}$; which is to say that, whenever $\alpha \colon H \Rightarrow GF$ is a $2$-cell in  $\cat{rCat}$  with $F$ localic, there exists a unique $J \colon \D \to \E$ and $2$-cell $\beta \colon J \Rightarrow G$ with $JF = H$ and $\beta F = \alpha$, so that
\begin{equation*}
\cd{
\C \ar[d]_F \ar[r]^{H} \dtwocell[0.3]{dr}{\alpha} & \E\\
\D \ar[ur]_-{G} & {}
} =
\cd{\C \ar[d]_F \ar[r]^{H} \dtwocell[0.53]{dr}{\beta} & \E\rlap{ .}\\
\D   \ar@{..>}@/^0.5pc/[ur]^-{J} \ar@/_1pc/[ur]_-{G} & {}
}
\end{equation*}
\end{Prop}
\begin{proof}
On objects, we define $JX = H\tilde X$ where, as before, $\tilde X$ is the unique object of $\C$ with $F\tilde X = X$. We take $\beta$ to have components $\beta_X = \alpha_{\tilde X} \colon JX = H\tilde X \to GF \tilde X = GX$. To define $J$ on morphisms, given $f \colon X \to Y$ in $\D$, we let $e = \overline{Gf.\beta_X}$, choose some $h \colon \tilde X \to \tilde Y$ with $f \leqslant Fh$, and now define $Jf = Hh.e$. Note that $\overline{Gf.\beta_X} \leqslant \overline{GFh.\beta_X} = \overline{\beta_Y.Hh} = \overline{\overline{\beta_Y}.Hh} = \overline{Hh}$ (using again~\cite[Lemma 2.1(iii)]{Cockett2002Restriction}), whence $\overline{Jf} = e$. 
It follows as in the proof of Proposition~\ref{prop:localic-hyper-orthogonal} that the definition of $J$ is independent of the choice of $h$, that $JF = H$, and that $J$ is a functor. Clearly $\beta F = \alpha$, and it only remains to show that $\beta$ is in fact natural in $f$. Given $f$ as above, we have $Jf = Hh.e$ as before. But now 
$\beta_Y.Jf = \alpha_{\tilde Y}.Hh.e = GFh.\alpha_{\tilde X}.e = GFh.\beta_X.\overline{Gf.\beta_X} = GFh.\overline{Gf}.\beta_X = Gf.\beta_X$ as required.
\end{proof}
We thus obtain the following ``enhanced'' orthogonality property of localic and hyperconnected morphisms.
\begin{Cor}\label{cor:enhanced}
Given a 2-cell $\alpha: G H \Rightarrow K F$ in $\cat{rCat}$ with $F$ localic and $G$ hyperconnected, there is a unique $J \colon \D \to \E$ and $\beta \colon GJ \Rightarrow K$ with $JF = H$ and $\beta F = \alpha$, so that
\begin{equation*}
\cd{
\C \ar[r]^H \ar[d]_F \dtwocell{dr}{\alpha} & \E \ar[d]^G \\
\D \ar[r]_K  & \F 
} \quad = \quad 
\cd{
\C \ar[r]^H \ar[d]_F \dtwocell[0.66]{dr}{\beta} & \E \ar[d]^G \\
\D \ar[r]_K \ar@{..>}[ur]^J & \F \rlap{ .}
}
\end{equation*}
\end{Cor}
\begin{proof}
Apply Propositions~\ref{prop:localic-hyper-orthogonal} and~\ref{prop:codisccofib}.
\end{proof}

%To see this, let $\mathbf 1$ be the terminal restriction category, and let $\mathbf J$ be the free restriction category containing a restriction idempotent: it has one object with two endomorphisms $0$ and $1$, composition given by multiplication, and restriction structure $\bar f = f$. There is a unique restriction functor $H \colon \mathbf 1 \to \mathbf J$, and it is easy to see that the hyperconnected morphisms are precisely those which are orthogonal on the right to $H$. Since by~\cite[Corollary 2.4]{Cockett2002Restriction} the category of restriction categories is locally finitely presentable, it follows by the constructions of~\cite{Freyd1972Categories} that hyperconnected morphisms are the right class of a factorisation system on $\cat{rCat}$. We call the corresponding left class of morphisms \emph{localic}, by analogy with the topos case 

\section{The fundamental functor}\label{sec:fund}
The restriction category $\cat{Stab}^\op$ of Examples~\ref{exs:restriction}(vi) in fact plays a privileged role in the theory of restriction categories: a construction given in~\cite[\S4.1]{Cockett2002Restriction} shows that every restriction category $\C$ admits a canonical restriction functor $\o \colon \C \to \cat{Stab}^\op$, called the \emph{fundamental functor} of $\C$. The following proposition summarises the main points of the construction.
\begin{Prop}
Let $\C$ be a restriction category.
\begin{enumerate}[(i)]
\item For each $A \in \C$, the set $\o(A)$ of restriction idempotents on $A$ is a meet-semilattice, with top element $1_A \colon A \to A$ and meet $e \wedge e' = ee'$.
\item For each map $f \colon A \to B$ in $\C$, the function $\om f \colon \o(B) \to \o(A)$ given by  $e \mapsto \overline{ef}$ is a stable map of meet-semilattices.
\item The assignations $A \mapsto \o(A)$ and $f \mapsto \om f$ are the action on objects and morphisms of a functor $\o \colon \C \to \cat{Stab}^\op$, the \emph{fundamental functor} of~$\C$.
\item The fundamental functor is a restriction functor.
\end{enumerate}
\end{Prop}

%We now give a universal characterisation of the fundamental functor.

\begin{Prop}\label{prop:fundterm}
The fundamental functor $\o \colon \C \to \cat{Stab}^\op$ of a restriction category $\C$ is a terminal object of $\cat{rCat}(\C, \cat{Stab}^\op)$.
\end{Prop}
\begin{proof}
We construct, for each restriction functor $F \colon \C \to \cat{Stab}^\op$,  a total natural transformation $\gamma \colon F \to \o$. Its component $\gamma_A \colon FA \leftarrow \o(A)$ is given by $\gamma_A(e) = Fe(\top)$; this is top-preserving, since $\gamma_A(\top) = (F1_A)(\top) = \top$, and binary-meet-preserving, since $\gamma_A(e \wedge e') = \gamma_A(ee') = Fe'(Fe(\top)) = Fe'(\top) \wedge Fe(\top)$, the last equality holding because $Fe'$ is a restriction idempotent. To show naturality, let $f \colon B \to A$ in $\C$; then $Ff(\gamma_A(e)) = Ff(Fe(\top)) = \overline{Ff.Fe}(\top) = F(\overline{ef})(\top) = \gamma_B(\om f e)$ as required. To show the uniqueness of $\gamma$, observe that for any $\gamma \colon F \to \O$ and any  restriction idempotent $e \colon A \to A$ in $\C$, we have a naturality square as on the left in:
\begin{equation*}
\cd{
\o(A) \ar[d]_{\om e} \ar[r]^{\gamma_A} & F(A) \ar[d]^{Fe} & & \top \ar@{|->}[d] \ar@{|->}[r] & \top \ar@{|->}[d]\\
\o(A) \ar[r]_{\gamma_A} & F(A) & & e \ar@{|->}[r] & \gamma_A(e) = Fe(\top)\rlap{ .}
}
\end{equation*}
Evaluating at $\top \in \o(A)$, and observing that $\gamma_A(\top) = \top$, as $\gamma$ is total, we obtain the square on the right, which shows that necessarily $\gamma_A = Fe(\top)$. \end{proof}

By the preceding proposition, for any restriction functor $F \colon \C \to \D$, there is a unique total natural transformation $\varphi$ fitting into a triangle
\begin{equation}\label{eq:gammatot}
\cd{
\C \ar[rr]^F \ar[dr]_\o & \ltwocell[0.4]{d}{\varphi} & \D \rlap{ .}\ar[dl]^\o \\
& \cat{Stab}^\op
}
\end{equation}
The components of $\varphi$ are the mappings $\o(A) \to \o(FA)$ sending $e$ to $Fe$, so that $\varphi$ is invertible just when $F$ is a hyperconnected restriction functor.

\begin{Prop}
Let $\C$ be a restriction category, and $F \colon \C \to \cat{Stab}^\op$ a restriction functor. Then the following are equivalent:
\begin{enumerate}[(i)]
\item $F$ is hyperconnected;
\item $F$ is a terminal object in the category $\cat{rCat}(\C, \cat{Stab}^\op)$;
\item $F$ admits a (necessarily unique) isomorphism to the fundamental functor of $\C$.
\end{enumerate}
\end{Prop}
\begin{proof}
By the preceding result, the fundamental functor $\o$ is a terminal object of $\cat{rCat}(\C, \cat{Stab}^\op)$ and so there is a unique total transformation $\gamma \colon F \to \o$; asking this to be invertible is equivalent  both to (ii) and to (iii). To show that it is also equivalent to (i), we  decompose $\gamma$ as a composite  $2$-cell
\begin{equation*}
\cd[@!]{
\C \ar[rr]^F \ar[dr]_\o \ltwocell[0.44]{drr}{\varphi} &  \ltwocell[0.57]{dr}{\delta} & \cat{Stab}^\op  \ar@/_0.5em/[dl]_(0.44){\o} \ar@/^2em/[dl]^{\id} \\
& \cat{Stab}^\op & {}
}
\end{equation*}
where $\varphi$ is as above, and $\delta$ is obtained using terminality of the fundamental functor of $\cat{Stab}^\op$. For any $X \in \cat{Stab}^\op$, the component $\delta_X \colon X \leftarrow \o(X)$ is the map sending $\phi$ to $\phi(\top)$, and this is invertible, with inverse $\o(X) \leftarrow X$ sending $e$ to $e \wedge (\thg)$. Thus $\delta$ is invertible, from which it follows that $\gamma$ is invertible just when $\varphi$ is---which is to say, just when $F$ is hyperconnected. 
\end{proof}
The following result gives our first reformulation of the notion of restriction category; in and of itself it is of scant interest, but it prepares the ground for our second reformulation in the following section.
\begin{Thm}\label{thm:equiv1}
The $2$-category $\cat{rCat}$ of restriction categories is $2$-equivalent to the $2$-category $\cat{rCat}'$ whose objects are hyperconnected restriction functors $F \colon \C \to \cat{Stab}^\op$, whose morphisms are diagrams
\begin{equation*}
\cd[@!C]{
\C \ar[rr]^H \ar[dr]_F & \ltwocell[0.4]{d}{\gamma} & \D\ar[dl]^{G} \\
& \cat{Stab}^\op
}
\end{equation*}
in $\cat{rCat}$, and whose $2$-cells $(H, \gamma) \to (H', \gamma')$ are total natural transformations $\theta \colon H \to H'$ with $\gamma'.G\theta = \gamma$.
\end{Thm}
\begin{proof}
There is an obvious forgetful $2$-functor $U \colon \cat{rCat}' \to \cat{rCat}$. It follows easily from the fact that a hyperconnected morphism is terminal in its hom-category that $U$ is $2$-fully faithful; it is moreover surjective on objects, since every restriction category admits a hyperconnected morphism to $\cat{Stab}^\op$, namely, its fundamental functor.
\end{proof}
\begin{Rk}
Note that $\cat{rCat}'$ is a full sub-$2$-category of the lax slice $2$-category $\cat{rCat}\sslash\cat{Stab}^\op$ whose $0$-cells are hyperconected functors.  Thus we may write $\cat{rCat}'= \cat{rCat}\sslash_{\!\hbar\,}\cat{Stab}^\op$.  This occasions various remarks:
\begin{enumerate}[(1)] 
\item  $\cat{rCat}'$ is, in fact, a reflective sub-$2$-category of the lax slice $2$-category. Given a restriction functor $\C \to \cat{Stab}^\op$, its reflection into $\cat{rCat}'$ is the hyperconnected part of its (localic, hyperconnected) factorisation, whilst given a lax triangle over $\cat{Stab}^\op$, its reflection into $\cat{rCat}'$ is obtained using the enhanced orthogonality property of Corollary~\ref{cor:enhanced}. 
\vskip0.25\baselineskip
\item  It also follows from  Corollary~\ref{cor:enhanced} that the lax slice $\cat{rCat}\sslash \cat{Stab}^\op$ is $2$-equivalent to the full sub-$2$-category of $\cat{rCat}^\rightarrow$ (the \emph{strict} arrow category) whose objects are the localic morphisms.
\item For any restriction category $\cat{Y}$ we can form $\cat{rCat}\sslash_{\!\hbar\,}\cat{Y}$ as a reflective sub-$2$-category of the lax slice $2$-category.  The objects of this 2-category are then restriction categories with a ``fundamental functor'' to $\cat{Y}$:  this means its lattices of restriction idempotents ``live'' in the restriction category $\cat{Y}$.  Thus, for example, if $\cat{Y} = \cat{Top}_p$ then each idempotent semilattice $\o(X)$ would be identified with the locale of open sets of a topological space and, furthermore, each map would have to behave like a partial continuous map on these open sets.  
\end{enumerate}
\end{Rk}

\section{Restriction categories as local discrete fibrations}\label{sec:3}

%Suppose now that we are given a mere category $\C$ equipped with a functor $R \colon \C \to \cat{Stab}^\op$. By the preceding observations, there can be at most one restriction structure on $\C$ whose fundamental functor coincides, to within isomorphism, with $R$. We are naturally led to ask which functors $\C \to \cat{Stab}^\op$ arise in this way.
%
%We will tackle this problem by exploiting the two-dimensional aspects of a restriction category. 
As observed in the previous section, each restriction category may be viewed as a locally partially ordered $2$-category, in such a way that every restriction functor respects these local partial orders. We therefore have a forgetful $2$-functor
$\Gamma \colon \cat{rCat} \to \cat{2\text-Cat}
$
from the $2$-category of restriction categories to the $2$-category of $2$-categories. In this section, we study the relationship between this forgetful functor and the notions introduced in the previous section.

We first describe a class of maps in $\cat{2\text-Cat}$ which correspond to the hyperconnected morphisms in $\cat{rCat}$. Recall that a functor $p \colon \E \to \B$ is called a \emph{discrete fibration} if, for every $e \in \E$ and map $\gamma \colon b \to pe$ in $\B$, there exists a unique map $\gamma'$ in $\E$ with $p(\gamma') = \gamma$. A $2$-functor $F \colon \K \to \L$ between $2$-categories is called a \emph{local discrete fibration} if each functor $\K(X,Y) \to \L(FX, FY)$ is a discrete fibration.

\begin{Prop}\label{prop:hyperdisc}
A restriction functor $F \colon \C \to \D$ is hyperconnected if and only if $\Gamma F  \colon \Gamma\C \to \Gamma\D$ is a local discrete fibration.
\end{Prop}
\begin{proof}
Suppose first that $F$ is hyperconnected. We must show for each $A,B \in \C$ that the functor
$(\Gamma F)_{A,B} \colon (\Gamma\C)(A,B) \to (\Gamma\D)(FA,FB)
$ is a discrete fibration. Thus given $g \in \C(A,B)$ and $k \leqslant Fg$ in $\D(FA,FB)$, we must show that there is a unique $f \leqslant g$ in $\C(A,B)$ with $Ff = k$. Because $F$ is hyperconnected, there exists a unique  $e \in \O(A)$ with $Fe = \overline k$. Now take $f = ge$; clearly we have $f \leqslant g$, and moreover $Ff = Fg.Fe = Fg.\overline k = k$. Finally, if $f' \leqslant g$ with $Ff' = k$, then from $F(\overline{f'}) = \overline{Ff'} = \overline k = Fe$ we deduce that $\overline{f'} = e$, and so that $f' = g \overline{f'} = ge = f$, as required.

Suppose conversely that $\Gamma F$ is a local discrete fibration. We must show that for each $A \in \C$, the induced mapping $\o (A) \to \o(FA)$ is an isomorphism. So let $e \colon FA \to FA$ be a restriction idempotent. We have $e \leqslant F(1_A)$ and so a unique $e' \colon A \to A$ with $e' \leqslant 1_A$ and $Fe' = e$. We have $e' = 1_A.\overline{e'} = \overline{e'}$, so that $e'$ is a restriction idempotent over $e$; moreover, if $e''$ is another restriction idempotent over $e$ then $e'' \leqslant 1_A$ and $Fe'' = e$ imply that $e'' = e$ by uniqueness of liftings.
\end{proof}
\begin{Rk}\label{ex:comprehensive}
The discrete fibrations are the right class of the \emph{comprehensive} factorisation system~\cite{Street1973The-comprehensive} on $\cat{Cat}$, whose corresponding left class comprises the \emph{final} functors. As both discrete fibrations and final functors are stable under finite products, the comprehensive factorisation induces on $\cat{2\text-Cat}$ a factorisation system (bijective on objects and locally final, local discrete fibration). In light of the preceding result, applying this factorisation to maps in the image of $\Gamma$ yields the (localic, hyperconnected) factorisations described in Proposition~\ref{prop:loc-hyp-fact} above.
\end{Rk}

The importance of local discrete fibrations is that they allow us to lift restriction structure from the codomain to the domain:
\begin{Prop}\label{prop:disclift}
Let $\D$ be a restriction category, and let $F \colon \C \to \Gamma\D$ be a local discrete fibration. Then there is a unique restriction functor $\hat F \colon \hat \C \to \D$ such that $\Gamma\hat \C = \C$ and $\Gamma\hat F = F$.
\end{Prop}
\begin{proof}
First observe that, because $\Gamma\D$ is locally  partially ordered and $F$ is a local discrete fibration, $\C$ is also locally partially ordered. We take the underlying category of $\hat \C$ to be the underlying $1$-category of $\C$, and equip it with the following restriction structure. For each map $f \colon A \to B$ in $\C$, we have $\overline{Ff} \leqslant F(1_A) \in \D(FA,FA)$ and so, because $F$ is a local discrete fibration, have a unique $\bar f \leqslant 1_X$ in $\C(A,A)$ with $F(\bar f) = \overline{Ff}$. We now check the axioms (R1)--(R4).
\begin{enumerate}[(R1)]
\item $\bar f \leqslant 1_A$ implies $f\bar f \leqslant f$; and since $F(f \bar f) = Ff.F\bar f = Ff.\overline{Ff} = Ff$, we conclude by the uniqueness of liftings that $f\bar f = f$.\vskip0.5\baselineskip
\item $\bar f \leqslant 1_A$ and $\bar g \leqslant 1_A$ implies $\bar f \bar g \leqslant 1_A$ and  $\bar g \bar f \leqslant 1_A$. But $F(\bar g \bar f) = \overline{Fg}\,\overline{Ff} = \overline{Fg}\, \overline{Ff} = F(\bar f \bar g)$ and so by uniqueness of liftings, $\bar g \bar f = \bar f \bar g$.\vskip0.5\baselineskip
\item Now we have $\overline {g \bar f} \leqslant 1_A$ and $\bar g \bar f \leqslant 1_A$, but $F(\overline {g \bar f}) = \overline{Fg.\overline{Ff}} = \overline{Fg} . \overline{Ff} = F(\bar g \bar f)$, whence by uniqueness of liftings, $\overline{g \bar f} = \bar g \bar f$.\vskip0.5\baselineskip
\item Finally, we have $\bar g f \leqslant f$ and $f \overline{gf} \leqslant f$, but $F(\bar g f) = \overline{Ff}.Fg = Ff.\overline{Fg.Ff} = F(f \overline{gf})$, whence again by uniqueness of liftings, $\bar g f  = f \overline{gf}$.
\end{enumerate}
Thus $\hat \C$ is a restriction category as required. To show that $U\hat \C = \C$, we must check that $f \leqslant g$ in $\C(A,B)$ just when $g \bar f = f$. Now because $\bar f \leqslant 1_A$, we also have $g \bar f \leqslant g$, and so $g \bar f = f$ implies $f \leqslant g$. Conversely, if $f \leqslant g$, then also $Ff \leqslant Fg$, which is to say that $Fg .\overline{Ff} = Ff$. But now we have $f \leqslant g$ and also $g \bar f \leqslant g$, but $F(g \bar f) = Fg .\overline{Ff} = Ff$, whence, by uniqueness of liftings, $g \bar f = f$. Finally, it is clear that $F$ lifts to a restriction functor $\hat F \colon \hat \C \to \D$.

It remains to show that $\hat \C$ and $\hat F$ are unique over $\C$ and $F$. So suppose there is given some other restriction structure $f \mapsto \tilde f$ inducing the $2$-category structure of $\C$. Then for each $f \colon A \to B$, we have $\tilde f \leqslant 1_A$ and $\overline f \leqslant 1_A$; but now $F(\tilde f) = \overline {Ff} = F(\bar f)$ and so by uniqueness of liftings, we have $\tilde f = \overline f$. Thus $\hat \C$ is unique; the uniqueness of $\hat F$ is now immediate.
\end{proof}
The preceding two results allow us to reformulate the notion of hyperconnected restriction functor in purely $2$-categorical terms. The following result allows us to do similarly for the notion of total natural transformation.
\begin{Prop}\label{prop:disctotal}
Let $\C$ be a restriction category. A map $f \colon A \to B$ in $\C$ is total if and only if it is a discrete fibration in $\Gamma\C$.
\end{Prop}
Recall that a map $f$ in a $2$-category $\K$ is called a \emph{discrete fibration} if $\K(X,f)$ is one in $\cat{Cat}$, for each $X \in \K$.
\begin{proof}
If $f \colon A \to B$ is a discrete fibration in $\Gamma \C$, then in particular, composition with it reflects identity $2$-cells; whence from $\bar f \leqslant 1_A$ and $f . \bar f  = f  = f. 1_A$ we deduce that $\bar f = 1_A$, so that $f$ is total. Suppose conversely that $f$ is total. We must show that for every $b \colon X \to B$ and $a \colon X \to A$ with $b \leqslant fa$, there exists a unique $c \leqslant a$ with $fc = b$. Because $b \leqslant fa$, we have $b = fa \overline b$, so that taking $c = a \overline b$, the above two conditions are clearly satisfied. To show uniqueness of $c$, suppose that $d \leqslant a$ with $fd = b$. Then we have $\overline b = \overline{fd} = \overline{\overline f d} = \overline{d}$ by totality of $f$ and~\cite[Lemma 2.1(iii)]{Cockett2002Restriction} so that $d = a \overline{d} = a \overline b = c$ as desired.
\end{proof}

Combining the above results, we have:
\begin{Thm}\label{thm:equiv2}
The $2$-category $\cat{rCat}$ of restriction categories is $2$-equivalent to the $2$-category $\cat{rCat}''$ whose objects are local discrete fibrations $F \colon \C \to \Gamma\cat{Stab}^\op$; whose $1$-cells are diagrams
\begin{equation*}
\cd[@!C@-0.5em]{
\C \ar[rr]^H \ar[dr]_F & \ltwocell[0.4]{d}{\gamma} & \D \ar[dl]^{G} \\
& \Gamma\cat{Stab}^\op
}
\end{equation*}
with $H$ a $2$-functor and $\gamma$ a $2$-natural transformation whose components are discrete fibrations; and whose $2$-cells $(H, \gamma) \to (H', \gamma')$ are $2$-natural transformations $\theta \colon H \to H'$ with $\gamma'.G\theta = \gamma$.
\end{Thm}
\begin{proof}
It suffices to show that $\cat{rCat}''$ is $2$-equivalent to the $2$-category $\cat{rCat}'$ of Theorem~\ref{thm:equiv1}. By Proposition~\ref{prop:hyperdisc}, each hyperconnected morphism $F \colon \C \to \cat{Stab}^\op$ induces a local discrete fibration $\Gamma F \colon \Gamma \C \to \Gamma \cat{Stab}^\op$, and by Proposition~\ref{prop:disctotal}, this assignation provides the action on objects of a $2$-functor $\cat{rCat}' \to \cat{rCat}''$. By Propositions~\ref{prop:hyperdisc} and~\ref{prop:disclift}, this $2$-functor is surjective on objects; we claim it is also $2$-fully faithful. To show fully faithfulness on $1$-cells, consider a diagram
\begin{equation*}
\cd[@!C@-0.5em]{
\Gamma\C \ar[rr]^H \ar[dr]_{\Gamma F} & \ltwocell[0.4]{d}{\gamma} & \Gamma\D \ar[dl]^{\Gamma G} \\
& \Gamma\cat{Stab}^\op
}
\end{equation*}
in $\cat{2\text-Cat}$, with all components of $\gamma$ discrete fibrations. By Proposition~\ref{prop:disctotal}, each $\gamma_A$ is a total map. We must show that $H = \Gamma \hat H$ for a unique restriction functor $\hat H \colon \C \to \D$; it will then follow that $\gamma$ is the image under $\Gamma$ of the unique total natural transformation $G\hat H \to F$. Clearly $\hat H$ must be defined as $H$ is on objects and morphisms, and so all that is needed is to verify that $\overline{Hf} = H \overline f$ for each $f \colon A \to B$ in $\C$. Since $\overline{Hf}\leqslant 1_{HA}$ and $H \overline f  \leqslant 1_{HA}$, both are restriction idempotents on $HA$ in $\D$; since $G$ is hyperconnected, it suffices to show that $ G(\overline{Hf}) = GH\overline f$ in $\cat{Stab}^\op$. Observe first that,  for any map $k\colon X \to Y$ in $\C$, we have $\overline{GHk} = \overline{\overline{\gamma_Y}.GHk} = \overline{\gamma_Y.GHk} = \overline{Fk.\gamma_X}$, using naturality and totality of $\gamma$ together with~\cite[Lemma 2.1(iii)]{Cockett2002Restriction}.
Using this and the fact that $F$ and $G$ are restriction functors, we calculate that
\begin{align*}
G(\overline{Hf}) &= \overline{GHf} = \overline{Ff.\gamma_A} = \overline{\overline{Ff}.\gamma_A} = \overline{F\overline f.\gamma_A} = \overline{GH\overline f} = GH\overline f
%
%G(\overline{Hf})(\top) &= (\overline{GHf})(\top) = (GHf)(\top) = 
%(GHf)(\gamma_B(\top)) = \gamma_A(Ff(\top)) \\ &= \gamma_A(\overline{Ff}(\top)) = \gamma_A(F\overline f(\top)) = (GH\overline f)(\gamma_A(\top)) = (GH\overline f)(\top)
\end{align*}
in $\cat{Stab}^\op$, as required.
%
%
%
%Both are restriction idempotents on $GHA$ in $\cat{Stab}^\op$, and so it suffices to check that $(G\,\overline{Hf})(\top) = (GH\overline f)(\top)$. 
%But now using naturality and totality of $\gamma$, and the fact that $F$ and $G$ are restriction functors, we calculate that
%\begin{align*}
%G(\overline{Hf})(\top) &= (\overline{GHf})(\top) = (GHf)(\top) = 
%(GHf)(\gamma_B(\top)) = \gamma_A(Ff(\top)) \\ &= \gamma_A(\overline{Ff}(\top)) = \gamma_A(F\overline f(\top)) = (GH\overline f)(\gamma_A(\top)) = (GH\overline f)(\top)
%\end{align*}
%as required. 
This shows that $U \colon \cat{rCat}' \to \cat{rCat}''$ is fully faithful on $1$-cells.
As for $2$-cells, suppose we are given $\theta \colon (\Gamma H, \Gamma \gamma) \Rightarrow  (\Gamma H', \Gamma \gamma') \colon (\Gamma \C, \Gamma F) \to (\Gamma \D, \Gamma G)$ in $\cat{rCat}''$. We must show that $\theta = \Gamma \hat \theta$ for a unique $\hat \theta$, which will clearly be the case so long as $\theta$ has total components in $\D$. Since $G$ is hyperconnected, it reflects totality, and so it is enough to show that each $G\theta_X$ is a total map of $\cat{Stab}^\op$. But $\gamma'_X . G\theta_X = \gamma_X$ in $\cat{Stab}^\op$, and $\gamma_X$ and $\gamma'_X$ are both total, whence also $G\theta_X$ by the cancellativity properties of total maps.
%
%Because $F$ is a restriction functor, $F\overline f = \overline{Ff} = (Ff)(\top) \wedge (\thg)$ and so $(Ff)(\top) = (F\overline f)(\top)$. Similarly, because $G$ is a restriction functor, $(GHf)(\top) = (G\, \overline{Hf})(\top)$. But now by naturality and totality of $\gamma$, $(G\, \overline{Hf})(\top) = (GHf)(\top) = (GHf)(\gamma_B(\top)) = \gamma_A(Ff(\top) = \gamma_A(F\overline f(\top)) = (GH\overline f)(\gamma_A(\top)) = (GH\overline f)(\top)$.
\end{proof}

\section{Notions of enriched category}\label{sec:enrichment}
We are now in a position to explain how restriction categories may be viewed as enriched categories, but before doing so, must break off to discuss briefly the notion of enrichment we shall need. If $\V$ is a monoidal category, there is a well-known notion of \emph{category enriched in $\V$} or \emph{$\V$-category}~\cite{Kelly1982Basic}, whose homs are given by objects from $\V$ rather than by sets. It was first in the work of Walters~\cite{Walters1981Sheaves} that it was realised that a more general kind of enrichment was fruitful, in which the monoidal category $\V$ is replaced by a bicategory $\W$. In this case, a \emph{$\W$-category} $\C$ is given by:
\begin{itemize}
\item A set of objects $\ob \C$;
\item A function $\abs{\thg} \colon \ob \C \to \ob \W$;
\item For each $x, y \in \ob \C$, a $1$-cell $\C(x,y) \colon \abs x \to \abs y$ in $\W$;
\item For each $x \in  \ob \C$, a $2$-cell $j_x \colon 1_{\abs x} \Rightarrow \C(x,x)$ in $\W$; and
\item For each $x,y,z \in  \ob \C$, a $2$-cell $m_{xyz} \colon \C(y,z) \otimes \C(x,y) \Rightarrow \C(x,z)$ in $\W$
\end{itemize}
with $j$ and $m$ satisfying the usual associativity and unitality axioms; there are corresponding notions of $\W$-functor and $\W$-transformation. The use to which Walter put these concepts was in representing sheaves on a given locale (more generally, site) as categories enriched in an associated bicategory $\W$. Similarly, we will shortly be able to exhibit restriction categories as categories enriched in a particular bicategory.

However, it turns out that whilst we may identify restriction categories with $\W$-categories for a suitable $\W$, this correspondence does not extend to the functors and natural transformations between them; there are more restriction functors between two restriction categories than there are $\W$-functors between the corresponding $\W$-categories. To rectify this, we will consider enrichment over a yet more general kind of base, namely that of a \emph{weak double category}. Enrichments of this and even more general sorts were considered in~\cite{Leinster2002Generalized}.

Recall from~\cite{Grandis1999Limits} that a \emph{weak double category} $\mathbb W$ is a pseudo-category object in $\cat{Cat}$: it is given by collections of \emph{objects} $A, B, \dots$, \emph{vertical morphisms} $f, g, \dots \colon A \to B$, \emph{horizontal morphisms} $U,V, \dots \colon A \tor B$ and \emph{cells}
\begin{equation}\label{eq:typcell}
\cd{
  A \ar[d]_f \ar@{->}[r]|{\!|\!}^{U} \dtwocell{dr}{\alpha} & C \ar[d]^g \\
  B \ar@{->}[r]|{\!|\!}_{V}  & D\rlap{ ,}
}
\end{equation}
together with composition and identity operations for vertical and horizontal arrows, and for cells along vertical and horizontal boundaries. Composition of vertical arrows is strictly associative and unital, whilst that of horizontal arrows is only associative and unital up to \emph{globular cells}: ones whose vertical boundaries are identities. Every bicategory can be seen as a weak double category with only identity vertical arrows; conversely, the objects, horizontal arrows, and globular cells in any weak double category form a bicategory, the \emph{underlying bicategory} of the weak double category.

We now describe the notions of category, functor and natural transformation enriched in a weak double category $\mathbb W$.  Firstly, a $\mathbb W$-category is simply a category enriched over the underlying bicategory of $\mathbb W$. For $\mathbb W$-categories $\C$ and $\D$, a \emph{$\mathbb W$-functor} $\C \to \D$ is given by:
\begin{itemize}
\item A function $H \colon \ob \C \to \ob \D$;
\item Vertical morphisms $H_x \colon \abs{x} \to \abs{Hx}$ for each $x \in \ob \C$;
\item Cells
\begin{equation*}
\cd[@C+1em]{
  \abs x \ar[d]_{H_x} \ar[r]|{\!|\!}^{\C(x,y)} \dtwocell{dr}{H_{x,y}} & \abs y \ar[d]^{H_y} \\
  \abs {Hx} \ar[r]|{\!|\!}_{\D(Hx,Hy)} & \abs{Hy}
}
\end{equation*} for each $x, y \in \ob \C$, satisfying the usual two functoriality axioms.
\end{itemize}
A \emph{$\mathbb W$-natural transformation} $\alpha \colon H \Rightarrow K \colon \C \to \D$ is given by
a collection of cells
\begin{equation*}
\cd[@C+1em]{
  \abs x \ar[d]_{H_x} \ar[r]|{\!|\!}^{1_{\abs x}} \dtwocell{dr}{\alpha_{x}} & \abs x \ar[d]^{K_x} \\
  \abs {Hx} \ar[r]|{\!|\!}_{\D(Hx,Kx)} & \abs{Kx}  
}
\end{equation*}
satisfying one naturality axiom. The $\mathbb W$-categories, $\mathbb W$-functors and $\mathbb W$-natural transformations form a $2$-category $\mathbb W\text-\cat{Cat}$. In particular, when $\mathbb W$ is the weak double category associated to a bicategory $\W$, we obtain a $2$-category $\W$-$\cat{Cat}$ of $\W$-categories, $\W$-functors and $\W$-transformations, as in~\cite{Walters1981Sheaves}.

\section{Restriction categories as enriched categories}\label{sec:4}
We now return to the task of exhibiting restriction categories as enriched categories. The one remaining ingredient we require is a construction due to Brian Day~\cite{Day1970On-closed}. Given a locally small $2$-category $\K$, we consider the bicategory $\P \K$ whose objects are those of $\K$, and whose hom-categories are given by $\P \K(A,B) \defeq [\K(A,B)^\op, \cat{Set}]$. Writing $\Y$ for the Yoneda embedding into a presheaf category, the identity map in $\P \K(B,B)$ is the representable ${\mathcal Y}(1_B)$, and composition $\P \K(B,C) \times \P \K(A,B) \to \P \K(A,C)$ is Day convolution: it is determined by the requirement that it be cocontinuous in each variable and defined on representables by ${\mathcal Y}g . {\mathcal Y}f = {\mathcal Y}(g. f)$. 
%
%
%The identity map $1_B \in \P \K(B,B)$ is the image under Yoneda of the identity map $1_B \in \K(B,B)$, whilst the composition functor
%$(\P \K)(B,C) \times (\P \K)(A,B) \to (\P \K)(A,C)$
%is obtained by Day convolution, as the left Kan extension:
%\begin{equation*}
%\cd[@C+0.5em]{
%\K(B,C) \times \K(A,B) \ar[r]^-{\otimes} \ar[d]_{Y \times Y} \rtwocell{dr}{} & \K(A,C) \ar[d]^Y \\
%\P\K(B,C) \times \P\K(A,B) \ar@{-->}[r]_-{\otimes} & \P\K(A,C)\rlap{ .}
%}
%\end{equation*}
There is a homomorphism of bicategories ${\mathcal Y} \colon \K \to \P \K$ which is the identity on objects, and on each hom-category is the Yoneda embedding.
Using the equivalence between presheaves on a small category $\C$ and discrete fibrations over $\C$ in $\cat{Cat}$, it is now not difficult to derive the following result, due to Richard Wood in collaboration with the first author; we shall prove a generalisation of it as Proposition~\ref{prop:genwood} below.
\begin{Prop}[\cite{Wood2006Variation}]\label{prop:wood} 
If $\K$ is a locally small $2$-category, then the $2$-category $\P \K\text-\cat{Cat}$ is $2$-equivalent to the $2$-category $\cat{2}\text-\cat{Cat}\mathbin{/_{\mathrm{ldf}}} \K$ whose objects are local discrete fibrations $F \colon \C \to \K$ with $\C$ locally small, and whose $1$- and $2$-cells are $2$-functors and $2$-natural transformations  commuting with the projections to $\K$.
\end{Prop}
Comparing this result with Theorem~\ref{thm:equiv2}, we see that restriction categories may be identified with $\P(\Gamma\cat{Stab}^\op)$-enriched categories; but that, as anticipated above, restriction functors between restriction categories are rather more general than $\P(\Gamma\cat{Stab}^\op)$-functors between the corresponding $\P(\Gamma\cat{Stab}^\op)$-categories. To rectify this, we shall consider enrichment in a weak double category obtained by a double-categorical analogue of Day's construction.
 
Suppose that $\K$ is a locally small $2$-category; we construct a weak double category $\mathbb P\K$ as follows. Its objects are those of $\K$, a vertical arrow $A \to B$ is a discrete fibration $B \to A$ in $\K$, a horizontal arrow $A \tor B$ is a presheaf $U \in [\K(A,B)^\op, \cat{Set}]$, whilst a cell of the form~\eqref{eq:typcell} is a $2$-cell
\begin{equation}\label{eq:particularcell}
\cd{
  A \ar@{<-}[d]_{{\mathcal Y}f} \ar[r]^{U} \dtwocell{dr}{\alpha} & C \ar@{<-}[d]^{{\mathcal Y}g} \\
  B \ar[r]_{V}  & D
}
\end{equation}
in $\P \K$. Composition of vertical morphisms is as in $\K$, that of horizontal morphisms is as in $\P \K$ and cell composition is given by pasting in $\P \K$. It is easy to see that the underlying bicategory of $\mathbb P \K$ is isomorphic to $\P \K$.

%Suppose now that $\mathbb K$ is an op-equipment $(\thg)^\ast \colon \K_t^\op \to \K$ with $\K$ a locally small $2$-category, and $(\thg)^\ast$ a faithful functor. Suppose moreover that every arrow in the image of $(\thg)^\ast$ is a discrete fibration in $\K$; recall that a morphism $f$  is  a discrete fibration in $\K$ just when $\K(X,f)$ is one in $\cat{Cat}$ for each $X \in \K$.
%Let $\mathbb P\mathbb K$ denote the op-equipment
%\[\K_t^\op \xrightarrow{(\thg)^\ast} \K \xrightarrow{Y} \P \K\text.\]
\begin{Prop}\label{prop:genwood}
If $\K$ is a locally small $2$-category, then $\mathbb P \K\text-\cat{Cat}$ is $2$-equivalent to the $2$-category $\cat{2}\text-\cat{Cat}\mathbin{\sslash_{\mathrm{ldf}}} \K$ whose objects are local discrete fibrations $F \colon \C \to \K$ with $\C$ locally small; whose morphisms are diagrams
\begin{equation*}
\cd[@!C@-0.5em]{
\C \ar[rr]^H \ar[dr]_F & \ltwocell[0.4]{d}{\gamma} & \D \ar[dl]^{G} \\
& \K
}
\end{equation*}
with $H$ a $2$-functor and $\gamma$ a $2$-natural transformation whose components are discrete fibrations; and whose $2$-cells $(H, \gamma) \to (H', \gamma')$ are $2$-natural transformations $\theta \colon H \to H'$ with $\gamma'.G\theta = \gamma$.
\end{Prop}

To prove this result, we will need an alternate description of the cells in $\mathbb P \K$.
\begin{Lemma}\label{lem:otherdesc}
To give a cell~\eqref{eq:particularcell} is equally to give a functor  $\alpha$ fitting into a commutative diagram
\begin{equation}\label{eq:otherdesc}
\cd[@C+1em]{
\mathrm{el}\ U \ar[d]_{\alpha} \ar[rr]^-{\pi_U} && \K(A,C) \ar[d]^{\K(f, C)} \\
\mathrm{el}\ V \ar[r]_-{\pi_V}  & \K(B,D) \ar[r]_-{\K(B,g)}
& \K(B,C)\rlap{ .}
}
\end{equation}
\end{Lemma}
\begin{proof}
Because composition in $\P \K$ is Day convolution, to give a cell~\eqref{eq:particularcell} is equally to give a map $\Lan_{\K(f,C)}(U) \to \Lan_{\K(B,g)}(V)$ in $[\K(B,C)^\op, \cat{Set}]$. Now $\K(B,g)$ is a discrete fibration, since $g$ is, and thus the bottom face of~\eqref{eq:otherdesc} is the discrete fibration corresponding to $\Lan_{\K(B,g)}(V)$. To give a map $\alpha$ as indicated is equally to give a map from $\pi_U$ to the pullback of the bottom face along $\K(f,C)$; which is to give a map of presheaves $U \to \K(f,C)^\ast(\Lan_{\K(B,g)}(V))$, or equally, a map $\Lan_{\K(f,C)}(U) \to \Lan_{\K(B,g)}(V)$, as required.
\end{proof}
We are now ready to give:
\begin{proof}[Proof of Proposition~\ref{prop:genwood}]
We define a $2$-functor $\Lambda \colon \mathbb P \K\text-\cat{Cat} \to \cat{2}\text-\cat{Cat}\mathbin{\sslash_{\mathrm{ldf}}} \K$ as follows. For $\C$ a $\mathbb P \K$-enriched category, we let $\Lambda \C$ be the $2$-category whose objects are those of $\C$, and whose hom-category $(\Lambda\C)(x,y)$ is the category of elements of $\C(x,y) \colon \K(\abs x, \abs y)^\op \to \cat{Set}$. The composition and identities of $\Lambda \C$ are induced by the composition and identity cells of $\C$, and there is a local discrete fibration $\pi_\C \colon \Lambda \C \to \K$ which on objects, sends $x$ to $\abs x$ and on hom-categories is the discrete fibration $(\Lambda \C)(x,y) \to \K(\abs x, \abs y)$. This defines the action of $\Lambda$ on objects.

Now let $H \colon \C \to \D$ be a $\mathbb P \K$-functor. We thus have a function $H \colon \ob \C \to \ob \D$, discrete fibrations $H_x \colon \abs{Hx} \to \abs x$ in $\K$ for each $x \in \C$, and  $2$-cells
\begin{equation*}
\cd[@C+2em]{
  \abs x \ar@{<-}[d]_{{\mathcal Y}H_x} \ar[r]^{\C(x,y)} \dtwocell{dr}{H_{x,y}} & \abs y \ar@{<-}[d]^{{\mathcal Y}H_y} \\
  \abs {Hx} \ar[r]_{\D(Hx,Hy)} & \abs{Hy}
}
\end{equation*}
in $\P\K$ for each $x, y \in \C$. By Lemma~\ref{lem:otherdesc}, to give the $2$-cell $H_{x,y}$ is equally well to give a functor $(\Lambda H)_{x,y} \colon (\Lambda \C)(x,y) \to (\Lambda\D)(Hx, Hy)$ fitting into a commutative diagram
\begin{equation}\label{eq:nat}
\cd[@C+1.2em]{
(\Lambda \C)(x,y) \ar[d]_{(\Lambda H)_{x,y}} \ar[rr]^-{(\pi_\C)_{x,y}} && \K(\abs x,\abs y) \ar[d]^{\K(H_x, 1)} \\
(\Lambda\D)(Hx, Hy) \ar[r]_-{(\pi_\D)_{Hx,Hy}}  & \K(\abs{Hx},\abs{Hy}) \ar[r]_-{\K(1,H_y)}
& \K(\abs{Hx},\abs{y})\rlap{ .}
}
\end{equation}
We therefore obtain a $2$-functor $\Lambda H \colon \Lambda \C \to \Lambda \D$ which sends $x$ to $Hx$, and on hom-categories, is given by these functors $(\Lambda H)_{x,y}$; the $2$-functor axioms for $\Lambda H$ are implied by the $\mathbb P\K$-functor axioms for $H$. 
Moreover, the maps $H_x \colon \abs{Hx} \to \abs{x}$ are the components of a $2$-natural transformation $\gamma_H \colon \pi_D . \Lambda H \Rightarrow \pi_C$; the naturality of these components is expressed precisely by the commutativity of the diagrams~\eqref{eq:nat}. This defines the action of $\Lambda$ on morphisms.

Finally, let $\alpha \colon H \Rightarrow K \colon \C \to \D$ be a $\mathbb P \K$-natural transformation. Thus we have a family of $2$-cells
\begin{equation*}
\cd[@C+2em]{
  \abs x \ar@{<-}[d]_{{\mathcal Y}H_x} \ar[r]^{1_{\abs x}} \dtwocell{dr}{\alpha_x} & \abs x \ar@{<-}[d]^{{\mathcal Y}K_x} \\
  \abs {Hx} \ar[r]_{\D(Hx,Kx)} & \abs{Kx}
}
\end{equation*}
in $\P\K$, satisfying one naturality axiom. To give $\alpha_x$ is equally to give a morphism ${\mathcal Y}H_x \to \Lan_{\K(1, K_x)}(\D(Hx, Kx))$ in $[\K(\abs{Hx}, \abs x)^\op, \cat{Set}]$ and so, applying the Yoneda lemma and passing to categories of elements, a morphism $(\Lambda \alpha)_x$ fitting into a diagram
\begin{equation*}
\cd[@C+1.2em]{
 1 \ar@{=}[rr] \ar[d]_{(\Lambda \alpha)_x} && 1 \ar[d]^{H_x} \\
(\Lambda\D)(Hx, Kx) \ar[r]_-{(\pi_\D)_{Hx,Kx}}  & \K(\abs{Hx},\abs{Kx}) \ar[r]_-{\K(1,K_x)}
& \K(\abs{Hx},\abs{x})\rlap{ .}
}
\end{equation*}
The $\mathbb P\K$-naturality of $\alpha$ implies that the maps $(\Lambda \alpha)_x$ are components of a $2$-natural transformation $\Lambda H \Rightarrow \Lambda K \colon \Lambda \C \to \Lambda \D$; whilst the commutativity of the displayed rectangles implies that $\gamma_K . (\pi_D.\Lambda \alpha) = \gamma_H$. This defines the action of $\Lambda$ on $2$-cells.
It is now not hard to show that the $\Lambda$ so defined is essentially surjective on objects, and $2$-fully faithful, hence a $2$-equivalence.
%% To give $H_{x,y}$ is equally well to give a morphism 
%%\[\Lan_{\K(H_x,1)}(\C(x,y)) \to \Lan_{\K(1,H_y)}(\D(Hx,Hy))\] in $[\K(\abs{Hx}, \abs{y})^\op, \cat{Set}]$. Because $\K(1, H_y)$ is a discrete fibration, $\Lan_{\K(1,H_y)}$ may be computed at the level of categories of elements simply by postcomposing with $\K(1,H_y)$. It thus follows that to give $H_{x,y}$ is equally to give a morphism
%%\begin{equation*}
%%\cd{
%%\Gamma(X \ar[r]^{\alpha} \ar[dd]_{\pi} & \mathrm{el}\ Y \ar[d]^{\pi} \\
%%& \K(c,d) \ar[d]^{\K(1,k)} \\
%%\K(a,b) \ar[r]_{\K(h, 1)} & \K(c,b)\rlap{ .}
%%}
%%
%%
%
%category of elements of $(\D(Hx,Hy))$ is obtained from that of $\D(Hx,Hy)$ 
\end{proof}

If we now define $\mathbb R \defeq \mathbb P(\Gamma \cat{Stab}^\op)$, we immediately conclude from  Proposition~\ref{prop:genwood} and Theorem~\ref{thm:equiv2} that:
\begin{Thm}\label{thm:equiv3}
The $2$-category of restriction categories is $2$-equivalent to the $2$-category of $\mathbb R$-enriched categories.
\end{Thm}

\section{Join restriction categories}\label{sec:join}
We have shown that restriction categories may be seen as categories enriched over a certain base; by appropriately changing that base,  we will now see that certain variants of the notion of restriction category are also expressible as enriched categories. In this section we consider join restriction categories, which are restriction categories in which compatible families of parallel morphisms can be patched together. 

A parallel pair of maps $f, g \colon A \rightrightarrows B$ in a restriction category $\C$ are said to be \emph{compatible} if $f \overline g = g \overline f$; this says that $f$ and $g$ agree when restricted to their domain of mutual definition. A family of maps $(f_i \colon A \to B\mid i \in I)$ is called \emph{compatible} if pairwise so.
We call  $\C$ a \emph{join restriction category}~\cite{Cockett2011Differential,Guo2012Products} if every compatible family of maps in $\C(A,B)$ admits a join $\bigvee_i f_i$ with respect to the restriction partial order, and these joins satisfy the axioms:
\begin{enumerate}[(J1)]
\item $\overline{\bigvee_{i} f_i} = \bigvee_{i}\overline{f_i}$;
\item $(\bigvee_{i}f_i) g = \bigvee_{i}(f_ig)$ for all $g \colon A' \to A$ in $\C$;
\item $h(\bigvee_i f_i) = \bigvee_i(h f_i)$ for all $h \colon B \to B'$ in $\C$.
\end{enumerate}
In fact, it turns out that the third of these axioms is a consequence of the other two; see~\cite[Lemma 3.1.8]{Guo2012Products}. 
%The second axiom says that precomposition distributes over joins; that the same is true for postcomposition is in fact a consequence of these axioms.

\newcommand{\vstab}{{\cat{jStab}}}
\begin{Exs}\hfill
\begin{enumerate}[(i)]
\item The category $\cat{Set}_p$ is a join restriction category. A family of partial functions $(f_i \colon A \rightharpoonup B \mid  i \in I)$ is compatible if $f_i(a) = f_j(a)$ whenever $f_i$ and $f_j$ are both defined at $a$, and the union of such a family is defined by
\begin{equation*}
(\textstyle\bigvee_i f_i)(a) = \begin{cases} f_i(a) & \text{if there exists $i \in I$ with $f_i(a)$ defined;} \\ \text{undefined} & \text{otherwise.}\end{cases}
\end{equation*}
\item The category $\cat{Top}_p$ is a join restriction category; the joins in $\cat{Top}_p(A,B)$ are precisely what one needs to verify that the presheaf of $B$-valued continuous functions on $A$ is in fact a sheaf.
\vskip0.5\baselineskip
\item The category $\cat{Rec}$ of partial recursive functions is not a join restriction category, but it is in the obvious sense a \emph{finite} join restriction category. One may patch together finitely many compatible partial recursive functions by suitably interleaving their calculations; but the same is not true of an infinite compatible family unless that family is itself recursively indexed.\vskip0.5\baselineskip

\item Let $\vstab$ denote the subcategory of $\cat{Stab}$ whose objects are frames---complete posets verifying the infinite distributive law $a \wedge \bigvee_{i} b_i = \bigvee_{i} a \wedge b_i$---and whose morphisms $A \to B$ are maps preserving binary meets and all joins. If $A$ is a frame and $a \in A$, then it is easy to see that the restriction idempotent $a \wedge (\thg) \colon A \to A$ in $\cat{Stab}^\op$ lies in $\vstab^\op$; whence $\vstab^\op$ is a hyperconnected restriction subcategory of $\cat{Stab}^\op$. 
It is moreover a join restriction category. A family of maps $(f_i \colon A \leftarrow B)$ is compatible just when $f_i(b) \wedge f_j(\top) = f_j(b) \wedge f_i(\top)$ for all $i,j \in I$ and $b \in B$; the union of such a family is defined pointwise by $(\bigvee_i f_i)(b) = \bigvee_{i}f_i(b)$. 
\end{enumerate}
\end{Exs}
By a \emph{join restriction functor}, we mean a restriction functor $F \colon \C \to \D$ between join restriction categories that satisfies $F(\bigvee_i f_i) = \bigvee_{i}(Ff_i)$ for each compatible family of maps. The join restriction categories, join restriction functors and total natural transformations form a $2$-category $\cat{jrCat}$.

We will now exhibit this as $2$-equivalent to a $2$-category of enriched categories, following the same progression of transformations as in the preceding sections. We begin with a result which is a straightforward consequence of the join restriction axioms.
\begin{Prop}
Let $\C$ be a join restriction category. Then:
\begin{enumerate}[(i)]
\item For each $A \in \C$, $\o (A)$ is a frame.
\item For each $f \colon A \to B$ in $\C$, $\om f \colon \o (B) \to \o (A)$ preserves all joins.
\item The induced factorisation $\O \colon \C \to \vstab^\op$ of the fundamental functor of $\C$ is a hyperconnected join restriction functor.
\end{enumerate}
\end{Prop}
The next step is an analogue of Proposition~\ref{prop:fundterm}.
\begin{Prop}
If $\C$ is a join restriction category, then the fundamental functor $\O \colon \C \to \vstab^\op$ is a terminal object of $\cat{jrCat}(\C, \vstab^\op)$. 
\end{Prop}
\begin{proof}
Let $F \colon \C \to \vstab^\op$ be a join restriction functor; as in the proof of Proposition~\ref{prop:fundterm}, the total $\gamma \colon F \to \O$ must have components $\gamma_A \colon FA \leftarrow \O(A)$ given by $\gamma_A(e) = Fe(\top)$. The only extra point to check is that each component $\gamma_A$ preserves joins, for which we calculate that:
\begin{equation*}
\gamma_A(\textstyle\bigvee_i f_i) = F(\bigvee _i f_i)(\top) = (\bigvee_i Ff_i)(\top) = \bigvee_{i}(Ff_i(\top)) = \bigvee_{i}(\gamma_A(f_i))\text{ .}\qedhere
\end{equation*}
\end{proof}
It is now an identical argument to that of Section~\ref{sec:2} to conclude that:
\begin{Thm}\label{thm:joinequiv1}
The $2$-category $\cat{jrCat}$ of join restriction categories is $2$-equivalent to the $2$-category $\cat{jrCat}'=\cat{jrCat} \sslash_{\!\hbar\,\,} \cat{jStab}^\op$ whose objects are  hyperconnected join restriction functors $F \colon \C \to \vstab^\op$, whose morphisms are diagrams
\begin{equation*}
\cd[@!C]{
\C \ar[rr]^H \ar[dr]_F & \ltwocell[0.4]{d}{\gamma} & \D \ar[dl]^{G} \\
& \vstab^\op
}
\end{equation*}
in $\cat{jrCat}$, and whose $2$-cells $(H, \gamma) \to (H', \gamma')$ are total natural transformations $\theta \colon H \to H'$ with $\gamma'.G\theta = \gamma$.
\end{Thm}

To give our next reformulation of join restriction categories, we shall identify a notion to which they bear the same relation as do restriction categories to $2$-categories. By an \emph{fc-site}, we mean a category with finite connected limits equipped with a Grothendieck topology. By a \emph{$2$-fc-site}, we mean a $2$-category $\K$ whose hom-categories are fc-sites and for which pre- and postcomposition by $1$-cells preserves both covers and finite connected limits.
\begin{Prop}
If $\C$ is a join restriction category, then $\Gamma \C$ is a $2$-fc-site, where a family of maps $(f_i \leqslant f \mid i \in I)$ in $\Gamma\C(A,B)$ is covering just when  $\bigvee_i f_i = f$.
\end{Prop}
\begin{proof}
$\Gamma\C(A,B)$ is a poset and so trivially has equalisers. It also has  pullbacks: if $f, g \leqslant h$ then $f \times_h g$ is $f \bar g = g \bar f$. Now if $(f_i \leqslant f \mid i \in I)$ is a covering family in $\Gamma\C(A,B)$, then for any $g \leqslant f$,  the pullback family $(f_i \bar g \leqslant g \mid i \in I)$ is again  covering since $\bigvee_i (f_i \bar g) = (\bigvee_i f_i)\bar g = f \bar g = g$ as required. Thus each $\Gamma\C(A,B)$ is an fc-site.
It remains to show that whiskering by $1$-cells preserves covers and finite connected limits. Preservation of covers follows from (J2) and (J3); equalisers are trivially preserved; and as for pullbacks, these are clearly preserved by postcomposition, whilst precomposition preserves them by (R4).\end{proof}

%\begin{Prop}\label{prop:jointop}
%Let $\C$ be a join restriction category. For every $A, B \in \C$, the poset $\Gamma\C(A,B)$ bears a Grothendieck topology, in which a family of maps $(f_i \leqslant f \mid i \in I)$ is covering just when  $\bigvee_i f_i = f$.
%\end{Prop}
%In the statement of this result, note that for any family $(f_i \leqslant f \mid i \in I)$ in $\Gamma\C(A,B)$, the collection of maps $(f_i \mid i \in I)$ is compatible; in particular, its join always exists.
%\begin{proof}
%Note that $\Gamma \C(A,B)$ has pullbacks; if $f, g \leqslant h$ then $f \times_h g$ is given by $f \bar g = g \bar f$. Now if $(f_i \leqslant f \mid i \in I)$ is a covering family, then for any $g \leqslant f$,  the pullback family $(f_i \bar g \leqslant g \mid i \in I)$ is again  covering since $\bigvee_i (f_i \bar g) = (\bigvee_i f_i)\bar g = f \bar g = g$ as required.
%\end{proof}
If $\B$ is an fc-site, then we call a functor $p \colon \E \to \B$ \emph{\'etale} if it is a discrete fibration and the corresponding presheaf $\B^\op \to \cat{Set}$ is a $j$-sheaf. If $\K$ is a $2$-fc-site, then we call a $2$-functor $\L \to \K$ \emph{locally \'etale} if each induced functor on homs is \'etale.
\begin{Prop}\label{prop:joinhyperetale}
Let $F \colon \C \to \D$ be a hyperconnected restriction functor between join restriction categories. Then $F$ preserves joins if and only if $\Gamma F \colon \Gamma \C \to \Gamma \D$ is locally \'etale.
\end{Prop}
\begin{proof}
For each $A, B \in \C$, we know by Proposition~\ref{prop:hyperdisc} that $(\Gamma F)_{A,B} \colon \Gamma \C(A,B) \to \Gamma \D(FA,FB)$ is a discrete fibration. It will be \'etale precisely when 
for every cover $(g_i \leqslant g \mid i \in I)$ in $\Gamma D(FA, FB)$, each matching family of elements over it in $\C(A,B)$ admits a unique patching.
To give a matching family over the $g_i$'s is to give maps $(f_i \in \C(A,B) \mid i \in I)$ with $Ff_i = g_i$ for each $i$, and such that $f_{i}{\mathord\mid}_j = f_j{\mathord\mid}_i$ for all $i,j \in I$; here $f_i {\mathord\mid}_j$ is obtained, using the discrete fibration property of $\Gamma F$, as the unique element of $\C(A,B)$ with $f_i {\mathord\mid}_j \leqslant f_i$ and $F(f_i {\mathord\mid}_j) = g_i \times_g g_j$. We know from above that $g_i \times_g g_j = g_i \overline{g_j}$ and thus  $f_i {\mathord\mid}_j = f_i \overline{f_j}$; so to say that the $f_i$'s are a matching family is to say that they are compatible in the sense defined above.

Now let $F$ preserve joins. For any matching family $(f_i \mid i \in I)$ as above, the element $\bigvee_i f_i$ is a patching: it satisfies $f_i \leqslant \bigvee_i f_i$ for each $i$ and $F(\bigvee_i f_i) = \bigvee_i Ff_i = \bigvee_i g_i = g$. Moreover, if $h$ is another patching, so satisfying $f_i \leqslant h$ for each $i$ and $Fh = g$, then $\bigvee_i f_i \leqslant h$ and $Fh = g = F(\bigvee_i f_i)$ implies that $h = \bigvee_i f_i$, since discrete fibrations reflect identities. Thus $\Gamma F$ is locally \'etale.

Suppose conversely that $\Gamma F$ is locally \'etale, and let $(f_i \colon A \to B \mid i \in I)$ be a compatible family. We always have $\bigvee_i(Ff_i) \leqslant F(\bigvee_i f_i)$ and so it suffices to show the converse inequality. The $f_i$'s are a matching family for the cover $(Ff_i \leqslant \bigvee_i Ff_i \mid i \in I)$, and so admit a unique patching $k$ satisfying $Fk = \bigvee_i(Ff_i)$ and $f_i \leqslant k$ for each $i$. Thus $\bigvee_i f_i \leqslant k$ and so $F(\bigvee_i f_i) \leqslant \bigvee_i(Ff_i)$ as required.
\end{proof}
The following is now the analogue in this context of Proposition~\ref{prop:disclift}.
\begin{Prop}\label{prop:joindisclift}
Let $\D$ be a join restriction category, and let $F \colon \C \to \Gamma\D$ be locally \'etale. Then the
unique restriction functor $\hat F \colon \hat \C \to \D$ lifting $F$ is a join restriction functor between join restriction categories.
\end{Prop}
\begin{proof}
If $(f_i \colon A \to B \mid i \in I)$ are a compatible family in $\hat \C$, then they form a matching family for the cover $(Ff_i \leqslant \bigvee_i Ff_i \mid i \in I)$ in $\D$, and so admit a unique patching $f$ satisfying $Ff = \bigvee_i(Ff_i)$ and $f_i \leqslant f$ for each $i$. We claim that $f = \bigvee_i f_i$. Indeed, if $k \colon A \to B$ with $f_i \leqslant k$ for each $i$, then $\bigvee_i(Ff_i) \leqslant Fk$; now because $\Gamma F$ is a local discrete fibration, we obtain $h \leqslant k$ with $Fh = \bigvee_i(Ff_i)$. But $Ff_i \leqslant \bigvee_i(Ff_i)$ for each $i$ implies $f_i \leqslant h$ for each $i$, whence $h$ is a patching for the $f_i$'s. So $f = h \leqslant k$ as required.
The join restriction category axioms for $\hat \C$ now follow by an argument identical in form to that of Proposition~\ref{prop:disclift}. It is immediate that $\hat F$ is a join restriction functor.
\end{proof}

In our next result, we call a map $f$ in a $2$-fc-site $\K$ \emph{\'etale} if $\K(X,f)$ is an \'etale functor for each $X \in \K$.
\begin{Prop}\label{prop:ettotal}
Let $\C$ be a restriction category. A map $f \colon A \to B$ in $\C$ is total if and only if it is an \'etale map in $\Gamma\C$.
\end{Prop}
\begin{proof}
If $h \colon A \to B$ is \'etale in $\Gamma \C$, then it is in particular a discrete fibration, and so total by Proposition~\ref{prop:disctotal}. Conversely, if $h$ is total, then it is a discrete fibration in $\Gamma \C$; to show it is \'etale, let $X \in \C$ and consider $(g_i \leqslant g \mid i \in I)$ a cover in $\Gamma \C(X, B)$. Arguing as in Proposition~\ref{prop:joinhyperetale}, to give a matching family over the $g_i$'s is to give a compatible family $(f_i \in \C(X,A) \mid i \in I)$ with $hf_i = g_i$ for each $i$. We must show that there is a unique $f$ with $f_i \leqslant f$ for each $i$ and with $hf = g$. For existence, we may take $f = \bigvee_i f_i$; for uniqueness, observe that any $f$ with these properties satisfies $\bar{f} = \overline{hf} = 
\bar g = \bigvee_i \bar{g_i} = \bigvee_i \overline{hf_i} = \bigvee_i \overline{f_i}$ and also $\bigvee_i f_i \leqslant f$, whence $f = f \bar f = f \bigvee_i \overline{f_i} = \bigvee_i f_i$, as required.
\end{proof}
We can now give an analogue of Theorem~\ref{thm:equiv2}:
\begin{Thm}\label{thm:joinequiv2}
The $2$-category $\cat{jrCat}$ of join restriction categories is $2$-equivalent to the $2$-category $\cat{jrCat}''$ whose objects are locally \'etale maps $F \colon \C \to \Gamma(\vstab^\op)$; whose $1$-cells are diagrams
\begin{equation*}
\cd[@!C@-0.5em]{
\C \ar[rr]^H \ar[dr]_F & \ltwocell[0.4]{d}{\gamma} & \D \ar[dl]^{G} \\
& \Gamma(\vstab^\op)
}
\end{equation*}
with $H$ a $2$-functor and $\gamma$ a $2$-natural transformation with \'etale components, and whose $2$-cells $(H, \gamma) \to (H', \gamma')$ are $2$-natural transformations $\theta \colon H \to H'$ with $\gamma'.G\theta = \gamma$.
\end{Thm}
\begin{proof}
Arguing as in Theorem~\ref{thm:equiv2}, it suffices to construct a $2$-equivalence between $\cat{jrCat}''$ and the $\cat{jrCat}'$ of Theorem~\ref{thm:joinequiv1}. By Proposition~\ref{prop:joinhyperetale}, there is a $2$-functor $\cat{jrCat}' \to \cat{jrCat''}$, which by Propositions~\ref{prop:joinhyperetale} and~\ref{prop:joindisclift} is surjective on objects. To show $2$-fully faithfulness, consider a diagram
\begin{equation*}
\cd[@!C@-0.5em]{
\Gamma\C \ar[rr]^H \ar[dr]_{\Gamma F} & \ltwocell[0.4]{d}{\gamma} & \Gamma\D \ar[dl]^{\Gamma G} \\
& \Gamma(\vstab^\op)
}
\end{equation*}
with $\gamma$ a $2$-natural transformation whose components are discrete fibrations. Postcomposing with the local discrete fibration $\Gamma(\vstab^\op) \to \Gamma \cat{Stab}^\op$ and applying Theorem~\ref{thm:equiv2}, we conclude that $H$ lifts to a restriction functor $\C \to \D$; we must furthermore show that it preserves joins. Given a compatible family in $\C(A,B)$, we always have $\bigvee_i Hf_i \leqslant H(\bigvee_i f_i)$, so it suffices to show that $\overline {\bigvee_i Hf_i} = \overline{H(\bigvee_i f_i)}$, or equally that $\bigvee_i H\bar f_i = H(\bigvee_i \bar f_i)$. Since $G$ is hyperconnected, it suffices to prove this last equality on postcomposition with $G$; which follows from the calculation
\begin{equation*}
\textstyle G(\bigvee_i H\bar f_i) = \bigvee_i GH\bar f_i = \bigvee_i (F\bar f_i . \gamma_A) = (\bigvee_i F\bar f_i).\gamma_A = F(\bigvee_i \bar f_i).\gamma_A = GH(\bigvee_i \bar f_i)
%
%\textstyle G(\overline{H(\bigvee_i f_i)}) = GH(\overline{\bigvee_i f_i}) = F(\overline{\bigvee_i f_i}).\gamma_A = (\bigvee_i \overline{Ff_i}).\gamma_A = \bigvee_i(\overline{Ff_i}.\gamma_A) = \bigvee_i GH\bar{f_i}
\end{equation*}
in $\vstab^\op$, whose second and fifth equalities arise as in the proof of Theorem~\ref{thm:equiv2}. This proves fully faithfulness on $1$-cells; that on $2$-cells is as before.
\end{proof}

Finally, we are ready to show how join restriction categories may be viewed as enriched categories. In preparation for this, we recall a result of Brian Day allowing us to ``locally localise'' a bicategory.
\begin{Prop}[\cite{Day1973Note}]
Let $\L$ be a bicategory and $\Sigma$ a class of $2$-cells of $\L$ closed under whiskering with $1$-cells on each side. Suppose that for each $A,B \in \L$, the objects in $\L(A,B)$ which are orthogonal to $\Sigma \cap \L(A,B)$ constitute a full reflective subcategory $\L_\Sigma(A,B)$ with reflector $\ell_{A,B} \colon \L(A,B) \to \L_\Sigma(A,B)$, say.
Then there is a bicategory $\L_\Sigma$ with homs given as above and composition law
\begin{equation*}
\L_\Sigma(B,C) \times \L_\Sigma(A,B) \hookrightarrow \L(B,C) \times \L(A,B) \xrightarrow{\circ} \L(A,C) \xrightarrow{\ell} \L_\Sigma(A,C)\rlap{ .}
\end{equation*}
Moreover, the maps $\ell_{A,B}$ assemble into a homomorphism of bicategories $\ell \colon \L \to \L_\Sigma$ which exhibit $\L_\Sigma$ as the localisation of $\L$ with respect to the class of $2$-cells $\Sigma$.
\end{Prop}
The particular instance of this result which will be relevant for us is the following.
\begin{Prop}\label{prop:2site}
If $\K$ is a $2$-fc-site, then there is a bicategory $\cat{Sh}(\K)$ with the same objects as $\K$, and with hom-categories $\cat{Sh}(\K)(A,B) = \cat{Sh}(\K(A,B))$.
\end{Prop}
\begin{proof}
We apply the preceding result, taking $\L = \P \K$ and taking $\Sigma$ to comprise all $2$-cells inverted by sheafification in each hom. The result is immediate so long as we can show that $\Sigma$ is stable under whiskering by $1$-cells. It is enough to do so with respect to whiskering by representable $1$-cells, since every $1$-cell in $\P \K$ can be written as a colimit of representables; composition in $\P \K$ is cocontinuous in each variable; and $\Sigma$ is stable under colimits. 

Given $f \colon A \to B$ in $\K$, the precomposition functor $(\thg) \circ \Y f \colon \P \K(B,C) \to \P \K(A,C)$ is given by left Kan extension along $\K(f,C)$. We must show that this maps $\Sigma$ into itself. Equivalently, we may show that the composite
\[
\P \K(B,C) \xrightarrow{\Lan_{\K(f,C)}} \P \K(A,C) \xrightarrow{\ell} \mathrm{Sh}(\K(A,C))
\]
extends through $\cat{Sh}(\K(B,C))$. This latter category is equally well the localisation of $\P \K(B,C)$ at the covering sieves for its topology, so it suffices to show that the displayed composite inverts each covering sieve; or equally that $\Lan_{\K(f,C)}$ sends covering sieves to maps in $\Sigma$. By assumption, $\K(f,C)$ preserves finite connected limits, and thus so does $\Lan_{\K(f,C)}$. In particular it preserves monomorphisms; of course, it also preserves colimits, and so it preserves epi-mono factorisations.

Now if $\phi \rightarrowtail Yg$ is the sieve generated by a family $(\alpha_i \colon g_i \to g \mid i \in I)$ in $\K(B,C)$, then it is the second half of the epi-mono factorisation of $\sum_i Yg_i \to Yg$; thus its image under $\Lan_{\K(f,C)}$ is the second half of the epi-mono factorisation of $\sum_i Y(g_if) \to Y(gf)$, and hence a covering sieve, since $(\thg) \circ f$ preserves covers.
This proves that $\Sigma$ is stable under precomposition with $1$-cells; replacing $\K$ by $\K^\op$ proves the same for postcomposition.
\end{proof}
%We will apply this result to the $2$-category $\Gamma(\vstab^\op)$; we may do so by virtue of our next result.
%\begin{Prop}
%If $\C$ is a join restriction category, then $\Gamma \C$ is a $2$-category satisfying the hypotheses of Proposition~\ref{prop:2site}.
%\end{Prop}
%\begin{proof}
%As we have already observed, each $\Gamma\C(A,B)$ has pullbacks: if $f, g \leqslant h$ then $f \times_h g = f\bar g = g \bar f$. Postcomposition with $1$-cells trivially preserves these pullbacks, whilst precomposition preserves them by (R4). Since any poset has equalisers preserved by any monotone map, we conclude that each hom-category of $\Gamma \C$ has finite connected limits preserved on each side by whiskering. Moreover, by Proposition~\ref{prop:jointop}, each hom bears a Grothendieck topology, whose covers are, by (J2) and (J3), preserved by whiskering.
%\end{proof}
We may now prove an analogue of Proposition~\ref{prop:wood}, identifying $\cat{Sh}(\K)$-enriched categories with locally \'etale $2$-functors $\C \to \K$, and using this, we may identify join restriction categories with $\cat{Sh}(\Gamma(\vstab^\op))$-enriched categories. However, as before, the enriched functors between $\cat{Sh}(\K)$-enriched categories are too limited to capture the general join restriction functors. As before, we rectify this by passing to a suitable weak double category, and enriching over that.

Suppose, then, that $\K$ is a $2$-fc-site; we define a weak double category $\mathbb S\cat{h}(\K)$ as follows. Its objects are those of $\K$, a vertical arrow $A \to B$ is an \'etale map $B \to A$ in $\K$, a horizontal arrow $A \tor B$ is a sheaf $U \in \cat{Sh}(\K(A,B))$, whilst a cell of the form~\eqref{eq:typcell} is a $2$-cell
\begin{equation*}
\cd{
  A \ar@{<-}[d]_{\ell(\Y f)} \ar[r]^{U} \dtwocell{dr}{\alpha} & C \ar@{<-}[d]^{\ell(\Y g)} \\
  B \ar[r]_{V}  & D
}
\end{equation*}
in $\cat{Sh}(\K)$. Composition of vertical morphisms is as in $\K$, that of horizontal morphisms is as in $\cat{Sh}(\K)$ and cell composition is given by pasting in $\cat{Sh}(\K)$. 

An argument identical in form to that of Proposition~\ref{prop:genwood} now proves that:
%Suppose now that $\mathbb K$ is an op-equipment $(\thg)^\ast \colon \K_t^\op \to \K$ with $\K$ a locally small $2$-category, and $(\thg)^\ast$ a faithful functor. Suppose moreover that every arrow in the image of $(\thg)^\ast$ is a discrete fibration in $\K$; recall that a morphism $f$  is  a discrete fibration in $\K$ just when $\K(X,f)$ is one in $\cat{Cat}$ for each $X \in \K$.
%Let $\mathbb P\mathbb K$ denote the op-equipment
%\[\K_t^\op \xrightarrow{(\thg)^\ast} \K \xrightarrow{Y} \P \K\text.\]
\begin{Prop}\label{prop:joingenwood}
If $\K$ is a $2$-fc-site, then $\mathbb S \cat{h}(\K)\text-\cat{Cat}$ is $2$-equivalent to the $2$-category $\cat{2}\text-\cat{Cat}\mathbin{\sslash_{\textrm{\upshape l\'et}}} \K$ whose objects are locally \'etale $2$-functors $F \colon \C \to \K$ with $\C$ locally small; whose morphisms are diagrams
\begin{equation*}
\cd[@!C@-0.5em]{
\C \ar[rr]^H \ar[dr]_F & \ltwocell[0.4]{d}{\gamma} & \D \ar[dl]^{G} \\
& \K
}
\end{equation*}
with $H$ a $2$-functor and $\gamma$ a $2$-natural transformation with \'etale components; and whose $2$-cells $(H, \gamma) \to (H', \gamma')$ are $2$-natural transformations $\theta \colon H \to H'$ with $\gamma'.G\theta = \gamma$.
\end{Prop}
Thus, defining $\cat j\mathbb R \defeq \mathbb S\cat h(\Gamma(\vstab^\op))$, we immediately conclude from  Proposition~\ref{prop:joingenwood} and Theorem~\ref{thm:joinequiv2} that:
\begin{Thm}\label{thm:joinequiv3}
The $2$-category of join restriction categories is $2$-equivalent to the $2$-category of $\cat j\mathbb R$-enriched categories.
\end{Thm}

\section{Range restriction categories}\label{sec:range}
In this final section, we turn our attention from join restriction categories to \emph{range restriction categories}~\cite{Cockett2012Range}. A range restriction category is a restriction category $\C$ equipped with a \emph{range} operator assigning to each map $f \colon A \to B$ a map $\hat f \colon B \to B$, subject to the following four axioms:
\begin{enumerate}[(RR1)]
\item $\overline{\hat f} = \hat f$ for all $f \colon A \to B$;
\item $\hat f f = f$ for all $f \colon A \to B$;
\item $\widehat{\bar g f} = \bar g \hat f$ for all $f \colon A \to B$ and $g \colon B \to C$;
\item $\widehat{ g \hat f} = \widehat{gf}$ for all $f \colon A \to B$ and $g \colon B \to C$.
\end{enumerate}
A \emph{range restriction functor} $F \colon \C \to \D$ is a restriction functor which also satisfies $F(\hat f) = \widehat{Ff}$ for each arrow $f$ of $\C$. We write $\cat{rrCat}$ for the $2$-category of range restriction categories, range restriction functors and total natural transformations.

By (RR1), each map $\hat f$ in a range restriction category is a restriction idempotent, whose intended interpretation is as as the image of the map $f$. For instance, the category $\cat{Set}_p$ of sets and partial functions is a range restriction category, when equipped with the range operator which to a partial function $f \colon A \rightharpoonup B$ assigns the partial function $\hat f \colon B \rightharpoonup B$ with
\begin{equation*}
\hat f(b) = \begin{cases} b & \text{if $f(a) = b$ for some $a \in A$;} \\ \text{undefined} & \text{otherwise.}\end{cases}
\end{equation*}

Although we have presented it as extra structure, having a range operator is in fact a property of a restriction category:
\begin{Lemma}(cf. \cite[\S 2.11]{Cockett2012Range})
A restriction category bears at most one range operator.
\end{Lemma}
\begin{proof}
Suppose that $(\hat{\ })$ and $(\tilde{\ })$ are two such operators. Then
\[\hat f = \widehat{\tilde f f} = \tilde f \hat f = \hat f \tilde f = \widetilde{\hat f f} = \tilde f\ \text.\qedhere\]
\end{proof}

The crucial fact for our purposes is that the property of having a range operator can be expressed purely in terms of the fundamental functor. Given a poset $B$ and $b \in B$, we write $B / b$ for the downset $\{x \in B \mid x \leqslant b\}$. We define a stable map of meet-semilattices $g \colon A \to B$  to be \emph{open}~\cite[Definition 2.5]{Cockett2012Range} if, when viewed as a map $A \to B / g(\top)$, it possesses a left adjoint $f \colon B / g(\top) \to A$---which we call a \emph{local left adjoint} for $g$---and this satisfies \emph{Frobenius reciprocity}:
\begin{equation*}
f(g(a) \wedge b) = a \wedge f(b) \qquad \text{for all $a \in A$ and $b \leqslant g(\top) \in B$.}
\end{equation*}
It is now easy to check that identity maps are open, composites of open maps are open, and that 
if $f \leqslant g$ and $g$ is open, then so is $f$.
It follows that $\cat{oStab}^\op$, the opposite of the category of meet semilattices and stable open maps, is a hyperconnected sub-restriction category of $\cat{Stab}^\op$.%
%
%The category $\cat{Rec}$ of partial recursive functions is a range restriction category in a similar manner, since the image of any partial recursive function is an r.e.\ set. On the other hand, the category $\cat{Top}_p$ fails to be a range restriction category; intuitively, the problem is that the image of an open set need not be open.

\begin{Prop}(cf. \cite[Proposition 2.13]{Cockett2012Range})
A restriction category $\C$ admits a range operator just when its fundamental functor $\O \colon \C \to \cat{Stab}^\op$ factors through $\cat{oStab}^\op$.
\end{Prop}
\begin{proof}
First suppose that $\C$ is a range restriction category. Then for each $f \colon A \to B$ in $\C$, we define a local left adjoint $f_!$ to $f^\ast$ by $f_!(a) = \widehat{fa}$. To check adjointness, let $a \leqslant \overline f \in \O(A)$ and $b \in \O(B)$; we must show that  $a  \leqslant f^\ast(b)$ iff $f_!(a) \leqslant b$, i.e., that $\overline{bf} e = e$ iff $b\widehat{fa} = \widehat{fa}$. But if $\overline{bf} a = a$, then
\[
b\widehat{fa} 
=%\stackrel{\text{(RR3)}}{=} 
\widehat{bfa} 
=%\stackrel{\text{(R4)}}{=}
\widehat{f\overline{bf}a} = \widehat{fa}\rlap{ ;}
\]
and conversely, if $b\widehat{fa} = \widehat{fa}$, then since $a \leqslant \overline f$, we have
\[
a = \overline{f}a 
=%\stackrel{\text{(R3)}}{=} 
\overline{fa} 
=%\stackrel{\text{(RR2)}}{=} 
\overline{\widehat{fa} fa} = \overline{b\widehat{fa} fa} 
=%\stackrel{\text{(RR2)}}{=}
\overline{bfa} 
=%\stackrel{\text{(R3)}}{=}
\overline{bf}a\rlap{ .}\]
To check  Frobenius reciprocity, we calculate that, for $a$ and $b$ as above, we have
\[
f_!(f^\ast(b) \wedge a) = \widehat{f\overline{bf}a} 
=%\stackrel{\text{(R4)}}{=}
 \widehat{bfa}
=%  \stackrel{\text{(RR3)}}{=} 
   b\widehat{fa} = b \wedge f_!(a)\rlap{ ,}
\]
as required. This proves that if $\C$ admits a range operator, then its fundamental functor factors through $\cat{oStab}^\op$. 

Suppose conversely that the fundamental functor of $\C$ factors through $\cat{oStab}^\op$. We define a range operator by $\hat f = f_!(\bar f)$. Clearly this satisfies (RR1). 
For (RR2), note first that $\overline{f_!(\bar f) f} = f^\ast(f_!(f^\ast(\top))) = f^\ast(\top) = \bar f$; whence $\hat f f = f_!(\overline f) f = f 
\overline{f_!(\bar f) f} = f \bar f = f$ as required. For (RR3), let $f \colon A \to B$ and $g \colon B \to C$. It is easy to show that $\bar g_! \colon \O(B) / \bar g \to \O(B)$ is given by inclusion, and it follows that $(\bar g f)_!$ is just the restriction of $f_!$ to $\O(A)/(\bar g \wedge \bar f)$. So now
\begin{equation*}
\widehat{\overline g f} = (\bar g f)_!(\overline{\bar g f}) = f_!(\overline{\bar g f}) = f_!(f^\ast(\bar g)) = f_!(f^\ast(\bar g) \wedge \bar f) = \bar g \wedge f_!(\bar f) = \bar g \hat f
\end{equation*}
as required, where the penultimate equality uses Frobenius reciprocity. Finally, for (RR4), let  $f \colon A \to B$ and $g \colon B \to C$; it is again easy to show that $(g\hat f)_!$ is the restriction of $g_!$ to $\O(B) / (\hat f \wedge \bar g)$, and thus
\begin{equation*} 
\widehat{g \hat f} = (g\hat f)_!(\overline{g \hat f}) = g_!(\bar g \hat f) = g_!(\bar g \wedge f_!(\bar f)) = g_!(f_!(f^\ast(\bar g) \wedge \bar f)) = g_!(f_!(\overline{gf})) = \widehat{gf}
\end{equation*}
as required, where the fourth equality again uses Frobenius reciprocity.
\end{proof}
We now relate range restriction functors to the fundamental functor. A commutative square in $\cat{Stab}$ as on the left in
\begin{equation*}
\cd{
A \ar[d]_f \ar[r]^h & B \ar[d]^g \\
C \ar[r]_k & D
} \qquad \qquad 
\cd{
A \ar[d]_f \ar@{<-}[r]^-{\ h_!} & B/h(\top) \ar[d]^{g/h(\top)} \\
C \ar@{<-}[r]_-{\ k_!} & D/k(\top)
}
\end{equation*}
is called \emph{Beck-Chevalley} if $f$ and $g$ are total, $h$ and $k$ are open---with local left adjoints $h_!$ and $k_!$, say---and the square on the right above is also commutative in $\cat{Stab}$.
If $F, G \colon \C \to \cat{Stab}^\op$ are restriction functors taking values in open maps, then we call a total natural transformation $\gamma \colon F \Rightarrow G$ \emph{Beck-Chevalley} if all of its naturality squares, seen as appropriately oriented squares in $\cat{Stab}$, are Beck-Chevalley.
\begin{Prop}
A restriction functor $F \colon \C \to \D$ between range restriction categories preserves the range operator just when the associated total natural transformation $\varphi$ of~\eqref{eq:gammatot} is Beck-Chevalley.
\end{Prop}
\begin{proof}
The Beck-Chevalley condition says that, for all $f \colon A \to B$ in $\C$, we have commutativity on the left in:
\begin{equation*}
\cd[@C+1em]{
\o (A) / \bar f \ar[r]^-{\varphi_A / 
\bar f} \ar[d]_{f_!} & \o (FA) / \overline{Ff}  \ar[d]^{(Ff)_!} \\
\o (B) \ar[r]_-{\varphi_B} & \o (FB)
}\qquad\qquad
\cd{
a \ar@{|->}[r] \ar@{|->}[d] & Fa \ar@{|->}[d] \\
\widehat{fa} \ar@{|->}[r] & F(\widehat{fa}) = \widehat{Ff.Fa}
}
\end{equation*}
Evaluating at $a \leqslant \bar f \in \O(A)$ as on the right, this says that $\widehat{F(fa)} = F(\widehat{fa})$ for every $f \colon A \to B$ and $a \leqslant \bar f \in \O(A)$ in $\C$; which is easily equivalent to $F$'s preserving the range operator.
\end{proof}
From the preceding two results and Theorem~\ref{thm:equiv1}, we thereby conclude that:
\begin{Thm}\label{thm:rangeequiv1}
The $2$-category $\cat{rrCat}$ of range restriction categories is $2$-equivalent to the $2$-category $\cat{rrCat}'$ whose objects are hyperconnected restriction functors $F \colon \C \to \cat{Stab}^\op$ which factor through $\cat{oStab}^\op$, whose morphisms are diagrams \begin{equation*}
\cd[@!C]{
\C \ar[rr]^H \ar[dr]_F & \ltwocell[0.4]{d}{\gamma} & \D \ar[dl]^{G} \\
& \cat{Stab}^\op
}
\end{equation*}
in $\cat{rCat}$ with $\gamma$ Beck-Chevalley, and whose $2$-cells $(H, \gamma) \to (H', \gamma')$ are total natural transformations $\theta \colon H \to H'$ with $\gamma'.G\theta = \gamma$.
\end{Thm}

Since range restriction categories are special kinds of restriction categories, they are equally well special kinds of $\mathbb R$-enriched category, where we recall that $\mathbb R = \mathbb P(\Gamma \cat{Stab}^\op)$. What we will now show is that range restriction categories are in fact $\cat r \mathbb R$-enriched categories for a suitable sub-double category $\cat r \mathbb R \subset \mathbb R$, with this equivalence extending to the functors and natural transformations between them.

The objects and arrows of $\cat r \mathbb R$ will be the same as $\mathbb R$. A horizontal arrow $A \tor B$ of $\cat r \mathbb{R}$ will be such an arrow in $\mathbb R$---thus a presheaf on $\Gamma(\cat{Stab}^\op)(A,B)$---whose associated discrete fibration
$\mathrm{el}\ U \to \Gamma(\cat{Stab}^\op)(A,B)$ factors through $\Gamma(\cat{oStab}^\op)(A,B)$. 
A cell of $\cat r \mathbb R$ may be described as follows. By Lemma~\ref{lem:otherdesc}, a cell of $\mathbb R$ of the form~\eqref{eq:typcell} can be identified with a functor $\alpha$ fitting into a commutative diagram
\begin{equation*}
\cd[@C+1em]{
\mathrm{el}\ U \ar[d]_{\alpha} \ar[rr]^-{\pi_U} && \Gamma(\cat{Stab}^\op)(A,C) \ar[d]^{(\thg) \circ f} \\
\mathrm{el}\ V \ar[r]_-{\pi_V}  & \Gamma(\cat{Stab}^\op)(B,D) \ar[r]_-{g \circ (\thg)}
& \Gamma(\cat{Stab}^\op)(B,C)\rlap{ .}
}
\end{equation*}
Such a cell will lie in $\cat r \mathbb R$ if, firstly, $U$ and $V$ lie in $\cat r \mathbb R$, so that $\pi_U$ and $\pi_V$ both take their image in open maps; and secondly, whenever given $(k,u) \in \mathrm{el}\ U$ with $\alpha(k,u) = (k',u')$, say, the square
\begin{equation*}
\cd[@C+1em]{
C \ar[d]_{g} \ar[r]^{k} & A \ar[d]^{f} \\
D \ar[r]_{k'} & B
}
\end{equation*}
in $\cat{Stab}$ is Beck-Chevalley.

\begin{Prop}
$\cat r \mathbb R$ is a sub-double category of $\mathbb R$.
\end{Prop}
\begin{proof}
Because restriction idempotents in $\cat{Stab}^\op$ are open maps, it follows that horizontal identities of $\mathbb R$ are in $\cat r \mathbb R$. We next show that, given $U \colon A \tor C$ and $W \colon C \tor E$ in $\cat r \mathbb R$, the composite $W U$ is again in $\cat r \mathbb R$. By the formula for Day convolution, an object of $\mathrm{el}\ W  U$ lying over $m \in \Gamma(\cat{Stab}^\op)(A,E)$ is an element of the set \[\textstyle\int^{k,\ell} Uk \times W\ell \times \Gamma(\cat{Stab}^\op)(A,E)(m, \ell k)\] so that an object of $\mathrm{el}\ (W  U)$ can be represented by giving
\begin{equation}\label{eq:triple}
\begin{gathered}(k\colon A \leftarrow C,\, u \in Uk) \in \mathrm{el}\ U\text,\\
(\ell \colon C \leftarrow E, w \in W\ell) \in \mathrm{el}\ W\\
\text{and} \quad m \leqslant k\ell \text{ in $\cat{Stab}$}\rlap{ .}
%m \colon A \leftarrow E (k,u) \in \mathrm{el}\ U \quad \text{and} \quad (\ell,w) \in \mathrm{el}\ W \qquad \text{such that $m \leqslant \ell k$.}
\end{gathered}
\end{equation}
Because $U$ and $W$ are in $\cat r \mathbb R$, both $k$ and $\ell$ are open, whence also $k\ell$. But now $m \leqslant k \ell$ implies that $m$ is open. Thus $\mathrm{el}\ W  U \to \Gamma(\cat{Stab}^\op)(A,E)$ factors through $\Gamma(\cat{oStab}^\op)(A,E)$ so that $W U$ lies in $\cat r \mathbb R$ as required.

%Because the inclusion functor $\cat{oStab}^\op \to \cat{Stab}^\op$ is hyperconnected, the induced functor $J \colon \Gamma(\cat{oStab}^\op)(A,B) \hookrightarrow \Gamma(\cat{oStab}^\op)(A,B)$ is a discrete fibration; it follows that a horizontal arrow $A \tor B$ of $\mathbb R$, i.e., a presheaf on $\Gamma(\cat{Stab}^\op)(A,B)$, will lie in $\cat r \mathbb R$ just when it lies in the image of $\Lan_J$, i.e., just when it is a colimit of representables on open maps. Since open maps are closed under composition in $\cat{Stab}$, we conclude that horizontal arrows of $\cat r \mathbb R$ are closed under identities and composition in $\mathbb R$.
%
It remains to show that cells of $\cat r \mathbb R$ are closed under identities and composition. This is not hard to do for vertical and horizontal identities, and for composition along a horizontal boundary; we consider the case of composition along a vertical boundary in more detail. Given cells
\begin{equation*}
\cd{
  A \ar[d]_f \ar@{->}[r]|{\!|\!}^{U} \dtwocell{dr}{\alpha} & C \dtwocell{dr}{\beta} \ar[d]^g \ar@{->}[r]|{\!|\!}^{W} & E \ar[d]^h \\
  B \ar@{->}[r]|{\!|\!}_{V}  & D \ar@{->}[r]|{\!|\!}_{X} & F
}
\end{equation*}
in $\cat r \mathbb R$, we must show that their composite in $\mathbb R$ lies again in $\cat r \mathbb R$. %Certainly the horizontal arrows $W \circ U$ and $X \circ V$ lie in $\cat r \mathbb R$. To verify the remaining part of the condition, 
Now, given an object of $\mathrm{el}\ W  U$ of the form~\eqref{eq:triple}, its image under $\beta \alpha \colon \mathrm{el}\ W U \to \mathrm{el}\ X  V$ is obtained as follows.
If the images of $(k,u)$ and $(\ell, w)$ under $\alpha$ and $\beta$ are $(k', u')$ and $(\ell', w')$ respectively, then we have a commutative diagram
\begin{equation*}
\cd{
B \ar@{<-}[r]^-{k'} \ar@{<-}[d]_f & D \ar@{<-}[d]^{g} \ar@{<-}[r]^{\ell'} & F \ar@{<-}[d]^h \\
A \ar@{<-}[r]_-{k} & C \ar@{<-}[r]_\ell &  E
}
\end{equation*}
in $\cat{Stab}$, and a $2$-cell $ m \leqslant k \ell$. Because $h$ is a discrete fibration in $\cat{Stab}^\op$, we obtain from these data a morphism $m' \colon B \leftarrow F$ of $\cat{Stab}$ with $m' \leqslant k' \ell'$ and $m'h = fm$; explicitly, $m' = \overline{fm}.k'\ell'$. Now the image under $\beta \alpha$ of~\eqref{eq:triple} is the object of $\mathrm{el}\ (X V)$ represented by the data
\begin{equation*}
(k',u') \in \mathrm{el}\ V \text, \quad (\ell',w') \in \mathrm{el}\ X \quad \text{and} \quad m' \leqslant k' \ell'\rlap{ .}
\end{equation*}
To verify that $\beta \alpha$ lies in $\cat r \mathbb R$, we must show that the left-hand square of
\begin{equation*}
\cd{
E \ar[r]^{m\vphantom{\overline{fm}}} \ar[d]_h & A \ar[d]^f \\
F \ar[r]_{m'} & B
} \qquad = \qquad
\cd{
 E \ar[d]^h \ar[r]^\ell & C \ar[r]^k \ar[d]^{g} & A \ar[d]^f \ar[r]^{\overline{m}} & A \ar[d]^f \\
 F \ar[r]_{\ell'} & D \ar[r]_{k'} & B \ar[r]_{\overline {fm}} & B
}%\cd{
%B \ar[r]^{k'} \ar[d]_f & D \ar[d]^{g} \ar[r]^{\ell'} & F \ar[d]^h \\
%A \ar[r]_{k} & C \ar[r]_\ell &  E  
%}
\end{equation*}
is a Beck-Chevalley square in $\cat{Stab}$. This square decomposes as the composite of the squares on the right. All three of those squares are Beck-Chevalley, the left two by assumption and the right-hand one since $(\overline{fm})_!$ and $(\overline m)_!$ are given simply by inclusion. The result now follows from the easy observation that  Beck-Chevalley squares compose.
%To verify that $\beta \circ \alpha$ lies in $\cat r \mathbb R$, we must prove that the square 
%\begin{equation*}
%\cd[@C+1.5em]{
%E \ar[r]^h & F
%A / m(\top) \ar[d]_{m_!} \ar[r]^{f / m(\top)} & B / m'(\top) \ar[d]^{m'_!} \\
%E \ar[r]_{h} & F
%}\end{equation*}
%commutes in $\cat{Stab}$. But $m \leqslant \ell k$ implies $m = \ell k \bar m$, and similarly $m' = \ell' k' \bar m'$; whence we have
%\[
%hm_! = h\ell_! k_! {\bar m}_! = \ell_!.g / \ell(\top).k_!{\bar m}_!
%\]
%
%and similarly $m' \leqslant \ell' k'$, it is easy to see that $m_!$ is the composite of $k_! \ell_!$
%
%
%the above diagram, now rotated through ninety degrees, may be decomposed as
%\begin{equation*}
%\cd{
%A / m(\top) \ar[d]_{m_!} \ar[r]^f & B / m'(\top) \ar[d]^{m'_!} \\
%E \ar[r]_{h} & F
%}\end{equation*}
\end{proof}
Given the manner in which we have defined $\cat r\mathbb R$, it is now an immediate consequence of Theorems~\ref{thm:equiv2} and~\ref{thm:rangeequiv1} that:
\begin{Thm}\label{thm:rangeequiv2}
The $2$-category of range restriction categories is $2$-equivalent to the $2$-category of $\cat r \mathbb R$-enriched categories.
\end{Thm}
%
%We now draw together the preceding two results:
%\begin{Prop}
%If $F \colon \C \to U(\cat{Stab}^\op)$ is a local discrete fibration of $2$-categories, then the lifted restriction functor $\hat F \colon \hat \C \to \cat{Stab}^\op$ is naturally isomorphic to the fundamental functor $R \colon \hat \C \to \cat{Stab}^\op$.
%\end{Prop}
%\begin{proof}
%
%\end{proof}

%\begin{Ex}
%Let $\C$ be any restriction category. It follows easily from (R1) and (R3) that each map of the form $\bar f \colon A \to A$ in $\C$ is an idempotent; we call a map of the form $\bar f$ a \emph{restriction idempotent}. To each object $A \in \C$ we associate the set $R(A)$ of restriction idempotents on $A$; this bears a partial order given by $e \leqslant e'$ iff $ee' = e$, and the axioms for a restriction category quickly imply that $R(A)$ is a meet semilattice, with top element $1_A \colon A \to A$, and intersection $e \wedge e' = ee'$.
%
%Given a map $h \colon $
%
%\end{Ex}

%%%% Bibliography

\bibliographystyle{acm}

\bibliography{bibdata}

\end{document}